\RequirePackage{fix-cm}
\documentclass[smallextended]{svjour3}       
\usepackage{graphicx}
\usepackage{amsmath, amssymb}
\usepackage{amsfonts}
\usepackage{graphicx}
\usepackage{enumerate}
\usepackage[usenames, dvipsnames]{color}
\usepackage{mathtools}
\usepackage[square, numbers]{natbib}
\usepackage{ifthen}
\usepackage{tikz}
\usepackage{appendix}
\usepackage{verbatim}
\usetikzlibrary{decorations.pathmorphing}
\usetikzlibrary{calc}
\usepackage{hyperref}
\usepackage{amsmath}
\usepackage{amsfonts}
\usepackage{euscript}

\usepackage{float}
\usepackage{amscd}
\usepackage{mathrsfs}
\usepackage{tikz}
\usepackage{todonotes}
\usepackage{tikz-cd}
\usepackage{verbatim}
\usepackage{amssymb}

\usepackage{fancyhdr,lipsum}
\makeatletter

\newcommand{\D}{\mathbb{D}}
\newcommand{\E}{\mathbb{E}}
\newcommand{\F}{\mathscr{F}}

\newcommand{\N}{\mathbb{N}}
\newcommand{\Om}{\Omega}

\newcommand{\R}{\mathbb{R}}

\newcommand{\dist}{\operatorname{dist}}
\newcommand{\eps}{\varepsilon}
\newcommand{\f}{\mathfrak{f}}

\newcommand{\br}{\operatorname{br}}

\newcommand{\loc}{\operatorname{loc}}

\newcommand{\GWI}{\operatorname{GWI}}

\newcommand{\CBI}{\operatorname{CBI}}

\newcommand{\h}{\boldsymbol{\operatorname{ht}}}
\newcommand{\csn}{\boldsymbol{\operatorname{csn}}}

\newcommand\cev[1]{\overleftarrow{#1}}

\newcommand{\upto}{\uparrow}

\spdefaulttheorem{ass}{Assumption}{\bf}{}

\begin{document}

\title{The Gorin-Shkolnikov identity and its random tree generalization
}


\author{David Clancy, Jr.
}


\institute{D. Clancy, Jr.\\
	Dept. of Mathematics, University of Washington \\
              \email{djclancy@uw.edu}           
}

\date{\today}

\maketitle

\begin{abstract}
In a recent pair of papers Gorin and Shkolnikov \cite{GS_paper} and  Hariya \cite{Hariya} have shown that the area under normalized Brownian excursion minus one half the integral of the square of its total local time is a centered normal random variable with variance $\frac{1}{12}$. Lamarre and Shkolnikov generalized this to Brownian bridges \cite{GLS_RBB}, and ask for a combinatorial interpretation. We provide a combinatorial interpretation using random forests on $n$ vertices. In particular, we show that there is a process level generalization for a certain infinite forest model. We also show analogous results for a variety of other related models using stochastic calculus.
\end{abstract}

	\section{Introduction}	
Let $e = (e_t;t\in[0,1])$ denote a standard Brownian excursion, or, equivalently, a 3-dimensional Bessel bridge from 0 to 0 of unit duration, see \cite[Chapter XII]{RY}. This process is a semi-martingale with quadratic variation $t$, and so possesses a family of local times $L = (L_t^v; v\ge 0, t\in[0,1])$ which satisfies almost surely the occupation time formula:
\begin{equation*}
\int_0^t g(e_r)\,dr = \int_0^\infty g(v) L_t^v\,dv, \qquad \forall g\text{ bounded Borel measurable}. 
\end{equation*} In this particular case, there exists a jointly continuous modification of $L$ for which the above formula holds almost surely.

Using topics in random matrix theory Gorin and Shkolnikov \cite{GS_paper} obtained, as a corollary of one of their main results, the following interesting distributional identity:
\begin{equation} \label{eqn:gs_iden}
\int_0^1 e_t\,dt - \frac{1}{2} \int_0^\infty \left(L_1^v\right)^2\,dv \overset{d}{=} \mathscr{N}\left( 0,\frac{1}{12}\right).
\end{equation} 

Gorin and Shkolnikov were studying tri-diagonal matrices of the form \begin{equation}\label{eqn:MN}
M_N^\beta = \frac{1}{\sqrt{\beta}} \begin{pmatrix}
\sqrt{2}g_1 & \chi_{(N-1)\beta}&0 &\dots&\\
\chi_{(N-1)\beta} & \sqrt{2} g_2 & \chi_{(N-2)\beta} &0 &\dotsm\\
0& \chi_{(N-2)\beta}& \sqrt{2} g_3 & \ddots\\
\vdots& &\ddots &\ddots \\
&&&&\chi_{\beta}\\
&&&\chi_{\beta}&\sqrt{2}g_N
\end{pmatrix},\qquad \text{for } \beta>0, N\ge 1,
\end{equation}
where $\chi_{\beta},\chi_{2\beta},\dotsm, \chi_{(N-1)\beta}$ are independent $\chi$-distributed random variables which are indexed by their parameters and $g_1,\dots,g_N$ are independent standard normal random variables. The ordered eigenvalues, $\lambda_1^{(N)}\ge \lambda_2^{(N)}\ge\dotsm\ge \lambda_{N}^{(N)}$, of $M_N^\beta$ in \eqref{eqn:MN} are of particular interest for their connection to $\beta$-ensembles. The celebrated work of Dumitriu and Edelman \cite{DE_beta} shows that joint distribution of the eigenvalues has a density proportional to 
\begin{equation*}
\prod_{1\le i< j\le N} (x_i-x_j)^\beta \prod_{k=1}^N e^{-\beta x_k/4}.
\end{equation*} 

Ram\'{i}rez, Rider and Vir\'{a}g \cite{RRV.11} relate a rescaling and recentering of largest eigenvalues, i.e. the \textit{edge}, of $M^\beta_N$ to the eigenvalues of the so-called stochastic Airy operator. In particular they show for each fixed $k = 1,2,\dotsm$ that
\begin{equation*}
\Lambda_{j,N}:= N^{1/6} \left(2\sqrt{N} - \lambda_j^{(N)}\right), \qquad j\in \{1,2,\dotsm ,k\}
\end{equation*} converge jointly in distribution as $N\to\infty$ to eigenvalues $\Lambda_1\le \Lambda_2\le \dotsm$ of the operator
\begin{equation*}
\mathscr{H}^\beta f:= \left(-\frac{d^2}{dx^2} + x + \frac{2}{\sqrt{\beta}} W'_x \right)f, \qquad f(0) =0 ,\qquad f\in L^2(\R_+)
\end{equation*} where $W'$ is a white noise on $\R_+:=[0,\infty)$. In turn, Gorin and Shkolnikov showed the powers of the matrix $\frac{1}{2\sqrt{N}}M_N^\beta$ converge in a certain operator-theoretic sense to the semigroup generated by $\frac{1}{2}\mathscr{H}^\beta$. See Theorem 2.1 in \cite{GS_paper} for a proper formulation of this convergence.

Proposition 2.7 in \cite{GS_paper} relates the eigenvalues of $\mathscr{H}^\beta$ to the functional in \eqref{eqn:gs_iden} by
\begin{equation*}
\E\left[\sum_{j \ge 1} e^{T \Lambda_j /2} \right] = \sqrt{\frac{2}{\pi T^3}} \E \left[\exp\left(-\frac{T^{3/2}}{2} \left( \int_0^1e_t\,dt - \frac{1}{\beta} \int_0^\infty\left( L_1^v\right)^2\,dv\right) \right)\right].
\end{equation*} By comparing the Laplace transform in the right-hand side with the prior work of Okounkov \cite{Okounkov_Moduli} when $\beta = 2$, Gorin and Shkolnikov were able to prove the equality in distribution \eqref{eqn:gs_iden}.

Lamarre and Shkolnikov \cite{GLS_RBB} connected a \textit{spiked} version of the matrix model in \cite{GS_paper} to a reflected Brownian bridge as well. If we let $B^{|\br|} = (B^{|\br|}_t;t\in [0,1])$ denote a reflected Brownian bridge and let $L = (L_t^v; v\ge 0, t\in[0,1])$ denote its local time, then the work of \cite{GLS_RBB} relates eigenvalues of a spiked operator to the random variable
\begin{equation}\label{eqn:Abeta}
A_\beta = \sqrt{12}\left( \int_0^1 B^{|\br|}_t \,dt - \frac{1}{\beta} \int_0^\infty \left(L_1^v\right)^2\,dv\right).
\end{equation} Lamarre and Shkolnikov are able to give an explicit formulation of the moment generating function of $A_\beta$ only in the case $\beta =2$, which evaluates to 
\begin{equation*}
\E\left[A_2^{2n-1} \right] = -\frac{2^n(2n-1)!}{4(n-1)!}\sqrt{6\pi},\qquad n = 1,2,\dotsm,7.
\end{equation*} This leads Lamarre and Shkolnikov ``to believe that $A_2$ admits an interesting combinatorial interpretation." This work is devoted to giving one such combinatorial interpretation using random trees or random forests which can be generalized to the infinite forest models of \cite{Aldous_AF,Duquesne_CRTI}. We should mention that there are many relationships between the area of a Brownian excursion and limits appearing in enumerative combinatorics. For a good survey on such connections, Janson's survey \cite{Janson_ExcursionArea} is an excellent resource.

Shortly after Gorin and Shkolnikov posted their work on the arXiv, Hariya \cite{Hariya} gave a path-wise interpretation and proof of the normality result \eqref{eqn:gs_iden}. Hariya's proof relies heavily on the Jeulin identity \cite{Jeulin_Iden}. Define 
\begin{equation*}
H(x) = \int_0^x L_1^y\,dy = \int_0^1 1_{[e_t\le x]} \,dt,
\end{equation*} and let $H^{-1} = (H^{-1}(t);t\in[0,1])$ denote its right-continuous inverse. Jeulin's identity is the following identity in distribution
\begin{equation*}
\left(\frac{1}{2} L_1^{H^{-1}(t)};t\in[0,1]\right) \overset{d}{=} \left( e_t; t\in[0,1]\right).
\end{equation*} Lamarre and Shkolnikov \cite{GLS_RBB} use Hariya's idea and some results of Pitman \cite{Pitman_SDE} to show 
\begin{equation}\label{eqn:AbetaCond}
\left(\frac{1}{\sqrt{12}}A_2 \bigg| L_1^0 = x\right) \overset{d}{=} \mathscr{N}\left(-\frac{x}{4} , \frac{1}{12} \right).
\end{equation}
For more information on this time-change including its proof, see Jeulin's original work \cite[pg. 264]{Jeulin_Iden} or a random tree interpretation and generalization in \cite{AMP_Jeulin,M_SST, Pitman_SDE}. The paper \cite{AMP_Jeulin} conveys this transformation succinctly: The Jeulin identity ``can be roughly interpreted as the width of the layer of the tree containing [a] vertex [...] where the vertices are labeled [...] in breadth-first order."

Because of their connection to random matrix theory and the study of the stochastic Airy semigroup, the distributional properties of the random variables $A_\beta$ in \eqref{eqn:Abeta} and the random variables
\begin{equation*}
\int_0^1 e_t \,dt - \frac{1}{\beta} \int_0^\infty\left( L_1^v\right)^2\,dv
\end{equation*} for each $\beta>0$ are of interest. Apart from the case $\beta=2$, where tools from stochastic calculus are useful \cite{Hariya,GLS_RBB}, it is not obvious what approach to studying these random variables will be fruitful. In this article, we present a discrete forest model in order to understand these random variables which can yield a wide class of results involving the difference of an integral of a stochastic process and a constant multiple of an integral of its squared local time. See Corollaries \ref{cor:crtHeightIntro} and \ref{cor:crtHeight} below. Some of these techniques hold even outside of a Brownian regime and \textit{do not} rely on the stochastic calculus techniques of \cite{Hariya,GLS_RBB}. However, these stochastic calculus techniques can be used when we are inside the Brownian regime. We develop the techniques in \cite{Hariya,GLS_RBB} further with Theorems \ref{thm:gaussianStructure} and \ref{thm:bridgeLT}.

\subsection{Random tree and branching process interpretation}

We begin by giving the discrete interpretation of the Gorin-Shkolnikov identity in \eqref{eqn:gs_iden} and its generalization in \eqref{eqn:AbetaCond} in terms of two statistics on a random forests. We then show that these same statistics have a functional limit for a wide class of random forest models. 

Consider the vertex set $[n]:=\{1,2,\dotsm, n\}$. By a forest on $[n]$, we mean a cycle-free graph on the vertices $[n]$. We say that a forest $\f$ is \textit{rooted} if each connected component has a distinguished vertex called its root. Roots will be denoted by the letter $\rho$. We equip the forests with the graph distance, denoted by $\dist(-,-)$. Given a forest $\f$ on $[n]$ we can define two statistics on the graph, measuring height and width. The first statistic, denoted by $\h$, is the height of the vertex, i.e. the distance from the root. The second statistic counts the number of ``cousin'' vertices, i.e. vertices at the same height; we denote this by $\csn$. If $v$ is in the same component as the root $\rho$, then the two statistics are
\begin{equation}\label{eqn:htandcsn}
\h(v) = \dist(v, \rho),\qquad \csn(v) = \#\{w\in \f: \h(v) = \h(w), v\neq w\}.
\end{equation} The statistics are portrayed graphically in Figure \ref{fig:treestats} below.

\begin{figure}[h]
	\centering
	\includegraphics[width=0.3\linewidth]{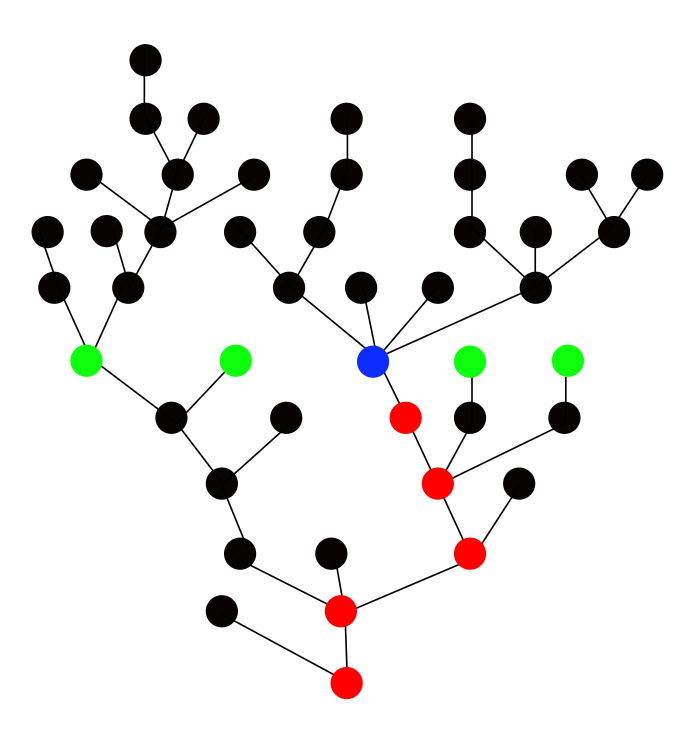}
	\caption{A pictorial representation of $\h$ and $\csn$. The green vertices are the cousins of the blue vertex, and the number of red vertices represents the height of the blue vertex. In this example, $\h(v) = 5$ and $\csn(v) = 4$ where $v$ is the blue vertex.}
	\label{fig:treestats}
\end{figure}

The following theorem describes the scaling relationship. 
\begin{theorem}\label{thm:excurForest}
	Fix a sequence $k_n$ such that $2k_n/\sqrt{n}\to x \ge 0$. For each $n$, let $\f^{(n)}$ denote a rooted forest on the vertex $[n]$ with $k_n$ roots uniformly chosen among all such forests. Then the following weak convergence holds as $n\to\infty$
	\begin{equation*}
	\frac{1}{2 n^{3/2} } \sum_{v\in \f^{(n)}} \h(v) - \frac{1}{n^{3/2}} \sum_{v\in \f^{(n)}} \csn(v) \overset{(d)}{\Longrightarrow} \mathscr{N}\left(\frac{-x}{4} , \frac{1}{12} \right).
	\end{equation*}
\end{theorem}

The above theorem relies on the connection between random trees and forests and excursions and reflected Brownian bridges in the literature. See, for example, \cite{Aldous_CRT1, Pitman_SDE} and references therein for more details on this connection. It also relies heavily on the extension of Jeulin's identity implicit in the work of Pitman \cite{Pitman_SDE}, which relates the local time of a reflected Brownian bridge with a time-change of a 3-dimensional Bessel bridge. We discuss some extensions of the stochastic calculus approach used in \cite{Hariya,GLS_RBB} below. Before moving to that section, we discuss a generalization using branching processes.

In the study of branching processes, there is a breadth-first model similar to the Jeulin identity. It is called the Lamperti transform originating in the work \cite{Lamperti_CSBP1}. Consider a genealogical structure with immigration depicted in Figure \ref{fig:intro} below.

\tikzset{
	treenode/.style={align=center, inner sep=0pt},
	node_black/.style={treenode, circle, white, draw=black, fill=black,  minimum size = .4cm},
	node_white/.style={treenode, circle, black, draw=black, minimum size = .4cm}
}
\begin{figure}[h!] 
	\begin{tikzpicture}[grow' = up, - ,level/.style = {sibling distance =.9 cm, level distance=.8cm}]
	\node {height 0}
	child{ node {height 1}  edge from parent[draw=none]
		child{ node {height 2}  edge from parent[draw=none]
			child{ node {height 3}  edge from parent[draw=none]}
		}	
	};
	\end{tikzpicture}
	\begin{tikzpicture}[grow' = up, - ,level/.style = {sibling distance =.9 cm, level distance=.8cm}]
	\node[node_black] {0}
	child{ node[node_black]{4} }
	;
	\end{tikzpicture}
	\begin{tikzpicture}[grow' = up, - ,level/.style = {sibling distance =.9 cm, level distance=.8cm}]
	\node[node_black]{1}
	child{ node[node_black]{5}
		child{ node[node_black]{13} }
		child{ node[node_black]{14} 
			child{ node[node_black]{20} }
		}
	};
	\end{tikzpicture}
	\begin{tikzpicture}[grow' = up, - ,level/.style = {sibling distance =.9 cm, level distance=.8cm}]
	\node[node_black] {2}
	child{ node[node_black]{6} 
		child{ node[node_black] {15}}
	}
	child{ node[node_black]{7} }
	child{ node[node_black]{8} 
		child{ node[node_black]{16} 
			child{ node[node_black]{21} }
			child{ node[node_black]{22} }
		}
	}
	;
	\end{tikzpicture}
	\begin{tikzpicture}[grow' = up, - ,level/.style = {sibling distance =.9 cm, level distance=.8cm}]
	\node[node_black] {3}
	child{ node[node_black]{9} 
		child{ node[node_black]{17} }
	}
	child {node[node_black]{10} }
	;
	\end{tikzpicture}
	\begin{tikzpicture}[grow' = up, - ,level/.style = {sibling distance =.9 cm, level distance=.8cm}]
	\node[node_white] {}
	child{ node[node_black]{11} }
	child{ node[node_black]{12} 
		child{ node[node_black]{18}}
		child{ node[node_black]{19}}
	}
	child{ node[node_white]{}
		child{ node[node_white]{}
			child{ node[node_black]{23} }
			child{ node[node_white]{} }	
		}
	}
	;
	\end{tikzpicture}
	\caption{The indices of the breadth-first labeling on an immigration forest of 24 non-mutant vertices.}
	\label{fig:intro}
\end{figure}
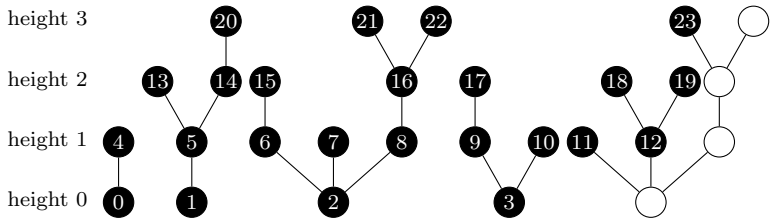

In this picture, the vertices (or individuals) come in two types. Using the language in \cite{Duquesne_CRTI}, the \textit{mutant} individuals are the unlabeled white vertices and the \textit{non-mutant} vertices are the labeled black vertices. The mutant vertices are simply a convenient way to introduce immigration and play little to no role in much of our analysis. The non-mutant vertices are labeled by $w_0,w_1,\dotsm$ in a breadth-first order with the convention that the immigrant vertices are labeled last in each generation. The mutant vertices are unlabeled. Let $\chi_j$ denote the number of children of the vertex $w_j$, and let $\eta_h$ denote the number of immigrants which arrive at height (or generation) $h$. In Figure \ref{fig:intro}, the forest has $4$ non-mutant individuals at height 0, the sequence of $\chi_j$'s begins $\chi_0 = 1$, $\chi_1 = 1$, $\chi_2 = 3$, etc. and the sequence of $\eta_h$'s begins $\eta_1 = 2$, $\eta_2 = 0$, and $\eta_3 = 1$. Conversely, given a number $k$ of non-mutant vertices at height $0$ (i.e. the roots) and two sequence of non-negative integers $(\chi_j; j =0,1 ,\dotsm)$ and $(\eta_h; h= 1,2,\dotsm)$ one can inductively construct a forest like the one in Figure \ref{fig:intro}. See Section \ref{sec:forest} for more information.

If we let $c_h$ denote the number of non-mutant vertices appearing at or before height $h$, then with the conventions $c_{-1} = 0$ the vertices at generation exactly $h$ are indexed by $c_{h-1} ,c_{h-1}+1,\dotsm, c_h -1$.  From here it follows
\begin{equation*}
c_{h+1} = c_h + \sum_{j={c_{h-1}}}^{c_h - 1} \chi_j + \eta_{h+1}.
\end{equation*} Let $z_h$ denote the number of non-mutant vertices at height exactly $h$, let $x_m = \sum_{j=0}^{m-1} (\chi_j-1)$ and let $y_h = \sum_{j=1}^h \eta_j$. As observed in \cite{CGU_CSBPI}, we can recover the successive generation sizes by solving the discrete equation
\begin{equation*}%
\label{eqn:discL}
z_{h+1} = k + x_{c_h} + y_{h+1} ,\qquad c_h = \sum_{j=0}^h z_j.
\end{equation*}
This is the \textit{discrete Lamperti transform}, which was rigorously studied in \cite{CGU_CSBPI} and generalized by the same authors in \cite{CPU_affine}. The authors described continuous time analogs and developed robust limit theorems. In particular, the authors of \cite{CGU_CSBPI} show that if $X = (X_t;t\ge0)$ is a L\'evy process with no negative jumps and $Y = (Y_t;t\ge 0)$ is an independent subordinator with Laplace transforms satisfying
\begin{equation*}
\E[\exp(-\lambda X_t)] = \exp(t\Psi(\lambda)),\qquad \E[\exp(-\lambda Y_t)] = \exp(-t\Phi(\lambda)) \,\qquad \lambda>0
\end{equation*} then there exists a unique c\`{a}dl\`{a}g solution to the equation
\begin{equation}\label{eqn:lamp}
Z_t = x+ X_{C_t} + Y_t,\qquad C_t = \int_0^t Z_s\,ds,
\end{equation} which exists until some possible explosion time. Moreover the process $Z = (Z_t;t\ge 0)$ is a continuous state branching process started from $x$ with branching mechanism $\Psi$ and immigration rate $\Phi$. We denote the law of $Z$ solving \eqref{eqn:lamp} for these L\'{e}vy processes $X$ and $Y$ by $\CBI_x(\Psi,\Phi)$. See Section \ref{sec:CBI} for more details on these processes. However, throughout this work we will work under the following assumptions on $\Psi$ and $\Phi$:
\begin{ass}\label{ass:lt}
	The function $\Psi$ is of the form
	\begin{equation*}
	\Psi(\lambda) = \alpha \lambda + \beta \lambda^2 + \int_{(0,\infty)} (e^{-\lambda r} - 1 + \lambda r 1_{[r<1]}) \,\pi(dr),
	\end{equation*} for some $\alpha\in \R$, $\beta\ge 0$ and a Radon measure $\pi$ such that $(1\wedge r^2)\pi(dr)$ is a finite measure.
	We assume that $\Psi$ in addition is \textit{conservative}, meaning \cite{Grey_explosion}
	\begin{equation}\label{eqn:cons}
	\int_{0}^\eps \frac{1}{|\Psi(u)|} \,du = \infty,\qquad \forall \eps>0.
	\end{equation}
	The subordinator $Y = (Y_t;t\ge 0)$ is $Y_t = \delta t$ for all $t\ge 0$, equivalently $\Phi(\lambda) = \delta\lambda$. \qed
\end{ass}

The first part of Assumption \ref{ass:lt} is equivalent to $X$ being a real-valued L\'{e}vy process which has no negative jumps and is not killed, see \cite{Bertoin_LP}. The second part of Assumption \ref{ass:lt} is in some sense the most general assumption on $\Psi$ that we can make for weak convergence arguments, see \cite{KW_BPI}.

As briefly mentioned above, we can construct a forest, $\f$, from two sequences $(\chi_j; j=0,1,\dotsm)$ and $(\eta_h; h =1,2,\dotsm)$ along with a non-negative integer $k$. We say that $\f$ is a \textit{Galton-Watson immigration forest} with offspring distribution $\mu$ and immigration distribution $\nu$ starting from $k$ roots if the sequences $(\chi_j)\overset{i.i.d.}{\sim}\mu$ and $(\eta_h)\overset{i.i.d.}{\sim} \nu$ and there are $k$ non-mutant individuals at height $0$. We use the notation $\f\sim \GWI_k(\mu,\nu)$. We also label the non-mutant vertices by $w_0,w_1,\dotsm$ in a breadth-first manner. See Section \ref{sec:forest} and Definition \ref{def:GWI} therein for a more in-depth description of this model.

We can again define two statistics, $\h$ and $\csn$ by \eqref{eqn:htandcsn}. We will examine the following two processes,
\begin{equation}\label{eqn:KJdef}
K_p^\f = \sum_{j=0}^{p-1} \csn(w_j),\qquad \text{and}\qquad J_p^\f = \sum_{j=0}^{p-1} \h(w_j).
\end{equation} We refer to $K^\f = (K^\f_p ; p = 0,1,\dotsm)$ as the cumulative breadth-first cousin process and $J^\f = (J^\f_p;p =0,1,\dotsm)$ as the cumulative breadth-first height process. 

We will be concerned about distributional limits of the processes $K^\f$ and $J^\f$, for sequences of forests $(\f^{(n)}; n \ge 1)$. For ease of notation we write $K^{\ast, n}$ instead $K^{\f^{(n)}}$, along with similar notation shifts for other processes.

Let $(\mu_n ;n \ge 1)$ and $(\nu_n; n\ge 1)$ be any sequence of probability measures on $\N_0 = \{0,1,\dotsm\}$ which satisfy Assumption \ref{ass:1} below.

\begin{ass}\label{ass:1}
	There exists a sequence $(\gamma_n; n \ge 1)$ of non-negative integers and a real-valued L\'{e}vy process $X = (X_t;t\ge 0)$ with non-negative jumps such that its negative Laplace exponent $\Psi$ is conservative (i.e. satisfies \eqref{eqn:cons}). We assume
	\begin{equation*}
	\frac{1}{n} \sum_{j=0}^{n\gamma_n -1} (\chi_j^{\ast,n} -1) \overset{(d)}{\Longrightarrow} X_1,\qquad \text{as }n\to \infty,
	\end{equation*} where $(\chi_j^{*,n}; j \ge 1)$ are i.i.d. with common distribution $\mu_n$. Simultaneously, for the same sequence $\gamma_n$, we suppose 
	\begin{equation*}
	\frac{1}{n} \sum_{j=1}^{\gamma_n} \eta_j^{\ast,n} \overset{p}{\longrightarrow} \delta,
	\end{equation*} for some non-random constant $\delta>0$ where $(\eta_j^{\ast,n};j\ge 1)$ are i.i.d. with common distribution $\nu_n$.
\end{ass}

There are many examples of sequences of distributions $(\mu_n;n\ge 1)$ and $(\nu_n;n\ge 1)$ which satisfy Assumption \ref{ass:1}. Some examples are included below. In the following examples and throughout the work, we write $[x]$ for the integer part of a real number $x$.
\begin{enumerate}
	\item For any $\mu_n = \mu$ with mean 1 and variance $\sigma^2<\infty$ and $\nu_n = \nu$ with mean $\delta>0$,
	\begin{equation*}
	\frac{1}{n} \sum_{j=0}^{n^2-1} \left(\chi_j^{\ast,n} - 1\right) \overset{(d)}{\Longrightarrow} \mathscr{N}(0,\sigma^2) ,\qquad\text{and}\qquad \frac{1}{n} \sum_{j=1}^{n} \eta_j^{\ast,n}\overset{p}{\longrightarrow} \delta,
	\end{equation*} by the central limit theorem and the weak law of large numbers. Hence, Assumption \ref{ass:1} holds for $\gamma_n = n$ and $X_t = \sigma B_t$ for a standard Brownian motion $B$ (i.e. $\Psi(\lambda) = \frac{\sigma^2}{2} \lambda^2$).
	\item For a non-centered normal distributions, take a fixed $a\in \R$ with $\mu_n = \operatorname{Poisson}(1+a n^{-1})$ for $n$ sufficiently large, $\nu_n = \nu$ having mean $\delta>0$ and $\gamma_n = n$, then
	\begin{equation*}
	\frac{1}{n} \sum_{j=0}^{n^2-1} \left(\chi_j^{*,n}-1\right) \overset{(d)}{\Longrightarrow} \mathscr{N}(a, 1),\qquad\text{and}\qquad \frac{1}{n}\sum_{j=1}^n \eta_j^{*,n} \overset{p}{\longrightarrow} \delta.
	\end{equation*} This corresponds to $X_t = B_t + a t$ and $\Psi(\lambda) = - a \lambda  +\frac{1}{2}\lambda^2$.
	\item Outside of a Brownian regime, we can take $\mu_n = \mu$ where $\lim_{k\to\infty} k^{1+\alpha}\mu(k) = c>0$ for some $\alpha\in(1,2)$ and $\sum_{k =0}^\infty k \mu(k) = 1$. Then, for $\gamma_n = [n^{\alpha-1}]$ we have
	\begin{equation*}
	\frac{1}{n} \sum_{j=0}^{n\gamma_n -1} \left(\chi_j^{*,n} -1 \right) \overset{(d)}{\Longrightarrow} X_1,
	\end{equation*} where $X$ is a spectrally positive $\alpha$ stable L\'{e}vy process and $\Psi(\lambda) = \tilde{c}\lambda^{\alpha}$ when $\lambda\ge 0$ for some constant $\tilde{c}$. See, for example, \cite[Section 3.7]{Durrett_PTE}. Taking, for example, $\nu_n = \operatorname{Poisson}(\delta n^{2-\alpha})$ for some $\delta>0$ implies
	\begin{equation*}
	\frac{1}{n}\sum_{j=1}^{\gamma_n} \eta_j^{*,n} \overset{p}{\longrightarrow} \delta.
	\end{equation*}
\end{enumerate}

We now state the following generalization of the Gorin-Shkolnikov identity, at least as described in the Theorem \ref{thm:excurForest} interpretation of the result. 
\begin{theorem}\label{thm:weakConvergence}
	Suppose that $\mu_n$ and $\nu_n$ satisfy Assumption \ref{ass:1} and fix some $x\ge 0$. For each $n$, let $\f^{(n)}$ denote a $\GWI_{[nx]}(\mu_n,\nu_n)$ forest. Let $K^{*,n}$ (resp. $J^{\ast,n}$) denote the cumulative breadth-first cousin (resp. height) process. Then, in the Skorohod space $\D(\R_+,\R)$ the following weak convergence holds
	\begin{equation*}
	\left(\frac{\delta}{n\gamma_n^2} J^{*,n}_{[n\gamma_nt]} - \frac{1}{n^2\gamma_n} K^{*,n}_{[n\gamma_nt]} ; t \ge 0 \right) \overset{(d)}{\Longrightarrow} \left(-xt - \int_0^t X_u\,du; t\ge 0 \right),
	\end{equation*} where $X$ is the L\'{e}vy process in Assumption \ref{ass:1}.
\end{theorem}

The above theorem relies on the following lemma.
\begin{lemma}\label{lem:integral1}
	Suppose that $Z$ is a $\CBI_x(\Psi,\Phi)$ process, where $\Psi$ and $\Phi$ satisfy Assumption \ref{ass:lt}. Define $V = (V_t; t \ge 0)$ by
	\begin{equation} \label{eqn:vdef}
	V_t = \inf\left\{r\ge 0 : \int_0^r Z_u\,du > t\right\}.
	\end{equation} Then
	\begin{equation*}
	\left(\int_0^{V_t}(Z_u-\delta u)\,Z_u\,du;t\ge 0 \right) \overset{d}{=} \left(xt + \int_0^t X_u\,du ;t\ge 0\right).
	\end{equation*}
\end{lemma}

It may not be immediately clear that Theorem \ref{thm:weakConvergence} is related to the integral identity in \eqref{eqn:gs_iden}. We hope to illuminate the connection with this example, based on the work by Duquesne \cite{Duquesne_CRTI}. Consider the sequences of measures $(\mu_n; n\ge 1)$ and $(\nu_n; n\ge 1)$ with $\mu_n = \mu =\operatorname{Poisson}(1)$ and $\nu_n= \nu = \operatorname{Poisson(\delta)}$ for each $n\ge 1$. For these measures Assumption \ref{ass:1} is clearly satisfied with the $\gamma_n = n$, and the L\'{e}vy process $X$ a Brownian motion $B$. 

We let $\f^{(n)}\sim \GWI_{[nx]}(\mu,\nu)$. As shown by Duquesne \cite{Duquesne_CRTI}, there exists an encoding (in a depth-first manner) of the non-mutant vertices in $\f^{(n)}$ by a height process $H^{n} = (H^n_j ; j = 0,1,\dotsm)$ such that
\begin{equation*}
\#\{w \in \f^{(n)}: \h(w) = h,\text{  $w$ is non-mutant}\} = \#\{j\ge 0 : H^{n}_j = h\}.
\end{equation*}
Moreover, Duquesne shows that 
\begin{equation*}
\left(\frac{1}{n} H^{n}_{[n^2t]} ;t \ge 0 \right) \overset{(d)}{\Longrightarrow} \left(\cev{H}_t; t \ge 0 \right),
\end{equation*} where 
\begin{equation*}
\cev{H}_t = 2 |W_t| + \frac{1}{\delta} \left(L^0_t(W) - x \right)_+,
\end{equation*} for a Brownian motion $W$ with its local time at level zero and time $t$ being $L^0_t(W)$ and $(x)_+:= \max(0,x)$.
Duquesne \cite{Duquesne_CRTI} also shows that the process $\cev{H}$ possesses a jointly measurable (jointly continuous in this situation \cite[Theorem VI.1.7]{RY}) family of local times $L(\cev{H}) = (L_t^y(\cev{H}); t\ge 0, y\ge 0)$ satisfying the occupation time formula
\begin{equation*}
\int_0^t g(\cev{H}_r)\,dr = \int_0^\infty g(y) L_t^y(\cev{H})\,dy,\qquad \forall g\in C_c(\R_+) \quad \text{a.s.}
\end{equation*} We remark that we \textbf{always} take this definition of local time even if the process is a semi-martingale with quadratic variation not identically $t$, which is the case in this present example. 

Then Theorem \ref{thm:weakConvergence}, Lemma \ref{lem:integral1} and the Ray-Knight theorem \cite{Duquesne_CRTI} imply the following corollary:
\begin{corollary}\label{cor:crtHeightIntro} Define $V = (V_t;t\ge 0)$ by \begin{equation*}
	V_t = \inf\left\{r\ge 0: \int_0^r L_\infty^y(\cev{H})\,dy > t\right\}.
	\end{equation*}
	Then
	\begin{equation}\label{eqn:midPage5FromEJP}
	\left(\delta \int_0^\infty \cev{H}_r 1_{[\cev{H}_r\le V_t]}\,dr - \int_0^{V_t} \left( L_t^y(\cev{H})\right)^2\,dy; t\ge0 \right)\overset{d}{=} \left(-xt - \int_0^t B_s\,ds;t\ge0 \right).
	\end{equation} In particular, for each $t$,
	\begin{equation*}
	\delta \int_0^\infty \cev{H}_r 1_{[\cev{H}_r\le V_t]}\,dr - \int_0^{V_t} \left( L_t^y(\cev{H})\right)^2\,dy \overset{d}{=} \mathscr{N}\left(-xt,\frac{t^3}{3} \right).
	\end{equation*} 
\end{corollary}
This corollary can be generalized further with Corollary \ref{cor:crtHeight}, which includes examples of processes $\cev{H}$ where the right-hand side is a spectrally positive L\'{e}vy process with certain additional constraints. The integral relationship in equation \eqref{eqn:midPage5FromEJP} is analogous to the Gorin-Skholnikov identity in the sense that both relate a linear combination of the area under a curve and the integral of the squared local time to a normal distribution. Of course, the formulation in Corollary \ref{cor:crtHeightIntro} gives a Gaussian process and not just a single normal distribution. This process-level generalization is further studied using stochastic calculus with Theorem \ref{thm:gaussianStructure}.

\subsection{Stochastic Calculus Approach}

While Corollary \ref{cor:crtHeightIntro} follows from Lemma \ref{lem:integral1} it can also be obtained by stochastic calculus without the appeal to Lamperti transform in \cite{CGU_CSBPI}. Similar methods were used by Hariya \cite{Hariya} and Lamarre and Shkolnikov \cite{GLS_RBB} to obtain the normality results in \eqref{eqn:gs_iden} and \eqref{eqn:AbetaCond}. We extend their methods in this paper. 

Suppose the L\'evy process in Lemma \ref{lem:integral1} is of the form $X_t = \sigma B_t + a t$ for $\sigma>0$ and $a\in \R$. Then an elementary calculation yields $\displaystyle xt+\int_0^t \left(\sigma B_s+ a s\right) \,ds$ is a Gaussian process with mean $\mu_t = xt+ \frac{a}{2}t^2$ and covariance function $\Gamma(t_1,t_2) = \sigma^2 \left(\frac{t_1^2 t_2}{2} - \frac{t_1^3}{6}\right)$ for $t_1\le t_2$. We can extend this to a more general Gaussian structure. To do this examine the stochastic differential equation:
\begin{equation}\label{eqn:ZsdeUC}
\begin{array}{ll}
dZ_v = g(C_v)\sqrt{Z_v}\,d W_v + \left(c+ f(C_v)Z_v  \right)\,dv,& Z_0 = x> 0\\
dC_v = Z_v\,dv& C_0 = 0 
\end{array}
\end{equation} where  $g$ is a function such that $g^2$ is Lipschitz and $0< \eps<g(t)<M<\infty$, $c$ is a constant such that $c \ge \frac{1}{2}\sup_t g(t)$ and $f$ is a continuous function and $W$ is a standard Brownian motion. In Section \ref{sec:ZUC} we prove the weak existence of a solution to \eqref{eqn:ZsdeUC} by a sequence of time-changes.

\begin{theorem}\label{thm:gaussianStructure} 
	Suppose that $(Z,C)$ is a weak solution to the stochastic differential equation in \eqref{eqn:ZsdeUC}. Let $V_t = \inf\{r\ge 0: C_r> t\}$. Define the process $X = (X_t;t\ge 0)$ by
	\begin{equation*}
	X_t = \int_0^{V_t} (Z_v - cv)Z_v\,dv.
	\end{equation*} 
	
	Then $X$ is a Gaussian process with continuous sample paths. Its mean is given by $$\displaystyle \mu_t = xt + \int_0^t (t-s)f(s)\,ds$$ and its covariance function is given by 
	$$
	\Gamma(t_1,t_2) = \int_0^{t_1} (t_2-s)(t_1-s) g^2(s)\,ds,\qquad t_1\le t_2.
	$$
\end{theorem}

We can also study stochastic differential equations of the form
\begin{align}
\label{eqn:Zsde} dZ_v &= a \sqrt{Z_v}\,dW_v + \left(c+f(C_v)Z_v - \frac{ Z_v^2}{1-C_v} \right)\,dv,  &Z_0=x\ge 0\\
\nonumber dC_v &= Z_v\,dv, &C_0=0,
\end{align} for a standard Brownian motion $W$, constants $c>\frac{1}{2}a^2$ and $f$ a continuous function. As shown by Pitman \cite{Pitman_SDE}, the terminal local time of a reflected Brownian bridge $(L_1^v;v \ge0)$ conditioned on the event $L_1^0 =x$ is a weak solution to the stochastic differential equation
\begin{equation*}
dZ_v = 2\sqrt{Z_v} \,dW_v + \left(4 - \frac{Z_v^2}{1- \int_0^v Z_s\,ds} \right) \,dv \qquad Z_0 = x.
\end{equation*} In Section \ref{sec:ZC} we prove the weak existence of a solution  to \eqref{eqn:Zsde}.
See also Leuridans' work \cite{L_RKConditioned} on the conditioned Ray-Knight theorem. In Section \ref{sec:intRel} we extend Lemma \ref{lem:integral1} to the SDEs in \eqref{eqn:Zsde}. We prove the following theorem:

\begin{theorem}\label{thm:bridgeLT}
	Suppose that $(Z,C)$ is a weak solution to the stochastic differential equation in \eqref{eqn:Zsde}. Then
	$$
	\int_0^\infty \left\{2Z_v^2 - cv Z_v \right\}\,dv \overset{d}{=} \mathscr{N}\left(\mu,\sigma^2 \right),
	$$ where $\displaystyle\mu =x+ \int_0^1 (1-s)f(s)\,ds$ and $\displaystyle\sigma^2 = \frac{a^2}{3}$.
\end{theorem}

The above theorem has the following corollary, due to \cite{GS_paper} when $x = 0$ and \cite{GLS_RBB} when $x>0$. See also \cite{Hariya}.
\begin{corollary}\label{cor:rbb} Let $B^{|\br|} = (B^{|\br|}_t;t\in[0,1])$ denote a reflected Brownian bridge and let $L = (L_t^v; t\in[0,1],v\ge 0)$ denote its local time profile. Then, conditionally on $L_1^0 = x\ge 0$, 
	\begin{equation*}
	\int_0^1 B^{|\br|}_t\,dt - \frac{1}{2}\int_0^\infty \left(L_1^v \right)^2\,dv \overset{d}{=} \mathscr{N}\left(-\frac{x}{4},\frac{1}{12} \right).
	\end{equation*}
\end{corollary}

\subsection{Another Example}

We do not claim that this is the best discrete interpretation of the Gorin-Shkolnikov identity found in \cite{GS_paper} and its generalization in \cite{GLS_RBB}. However, the model presented in this article presents a generalization to a wide class of random forests. There are other discrete models which are asymptotically related to Brownian excursions and Brownian bridges. We briefly make note of one such example, which may be of use in order to understand $A_\beta$ for other values $\beta\in (0,\infty)\setminus\{2\}$.

It is well known that a Brownian excursion can be constructed from a Brownian bridge by placing the origin at the bridges absolute minimum, see \cite{Vervaat}. One extension of the result can be found in the work of Chassaing and Janson \cite{CJ_vervaat}. In their analysis, they use another combinatorial object of study called parking functions. Certain classes of parking functions are in bijection with labeled trees on the vertex set $\{0,1,\dotsm, n\}$ rooted at 0. See \cite{CM_parking} and references therein for more details on the specifics of the bijection. In the discrete setting, the analog of the average height of a vertex in a rooted tree is the so-called displacement for parking functions, and the average number of cousin vertices is the average number of ``cars'' that, at some point, want to park in a certain parking space. In this work we do not endeavor to prove weak limits for these discrete analogs; however, we do state what the limiting objects should be. Going back to the belief that there may be an ``interesting combinatorial interpretation" of the quantity $A_2$, parking functions may be another way to find a combinatorial interpretation of the results. 

In this section we deviate from the notation in the last section. We let $B^{\br} = (B^{\br}_t;t\in[0,1])$ denote a Brownian bridge, and let $e = (e_t;t\in[0,1])$ denote a standard Brownian excursion. We take the (deterministic) periodic extensions of these functions so that they are defined on all of $\R$. Fix an $x\ge 0$, define the following processes
\begin{align*}
X^{(1)}_t &= B^{\br}_t - xt + \sup_{t-1\le s\le t}\left(xs-B^{\br}_s \right)\\
X^{(2)}_t &= e_t - xt + \sup_{0\le s\le t} \left(xs-e_s \right).
\end{align*} The process $X^{(2)}$ is just a reflected Brownian excursion with negative drift $-x$.

We now let $X$ denote either of the processes above. The occupation measure of $X$ is equal in law to the occupation measure of a reflected Brownian bridge conditioned on its local time at zero being $x/2$. See \cite[Cor. 2.3]{CJ_vervaat}. The presence of the factor of $\frac{1}{2}$ comes from the different conventions for the local time at 0 used in \cite{Pitman_SDE} and \cite{CJ_vervaat}. Since, almost surely the occupation measure for a reflected Brownian bridge has a continuous density with respect to the Lebesgue measure, the same holds for the occupation measure of $X$. Thus we get the following corollary of Corollary \ref{cor:rbb}.

\begin{corollary}\label{cor:drift}
	Let $X$ denote either $X^{(1)}$ or $X^{(2)}$ above. Let $(L_1^v;v\ge0)$ denote a continuous version of the density of the occupation measure for $X$. Then,
	$$
	\int_0^1 X_t\,dt - \frac{1}{2}\int_0^\infty (L_1^v)^2\,dv \overset{d}{=} \mathscr{N}\left(-\frac{x}{8},\frac{1}{12}\right).
	$$
\end{corollary}

Parking functions and random forests are not the only combinatorial models that have asymptotics related to the Brownian excursion or Brownian and their local times. As mentioned previously, numerous quantities in graph enumerations are related to either a Brownian excursion or its integrals, see \cite{Janson_ExcursionArea} and references therein. See also \cite{AP_BBAss} for the connection between these processes and the asymptotics of uniform random mappings.

\subsection{Overview of The Paper}

In Section \ref{sec:forest}, we discuss the discrete forest model underlying our weak convergence results in Theorem \ref{thm:weakConvergence}. This model is based on a random forest model used by Aldous \cite{Aldous_AF} and Duquesne \cite{Duquesne_CRTI}. Afterward, in Section \ref{sec:processesOnForests} we discuss in more detail the various processes defined on the random forest. In Sections \ref{sec:lp} and \ref{sec:CBI} we discuss L\'{e}vy processes and continuous state branching processes respectively. Importantly, we describe the path-by-path relationship between these two classes from \cite{CGU_CSBPI}. Section \ref{sec:heightIntro} contains information on $\Psi$-height processes and their Ray-Knight theorems.

After discussing these preliminaries we move to prove integral relationships for L\'{e}vy processes and their relationships with continuous state branching processes with immigration in Section \ref{sec:BPFacts}. This brief discussion allows us to prove Corollary \ref{cor:crtHeight}, which is the continuum random tree generalization of the Gorin-Shkolnikov identity.

In Section \ref{sec:weakConvergence} we prove our weak convergence results which rely on the integral results in Section \ref{sec:BPFacts}. The proofs of Theorems \ref{thm:excurForest} and \ref{thm:weakConvergence} are found in Section \ref{sec:someProofs}. In Section \ref{sec:intRel} we prove the normality results in Theorems \ref{thm:gaussianStructure} and \ref{thm:bridgeLT} after studying the properties of solutions to \eqref{eqn:ZsdeUC} or \eqref{eqn:Zsde} in Section \ref{sec:SDEstuff}.

\subsection*{Acknowledgments} The author thanks David Aldous for a helpful discussion on the random tree interpretation of Jeulin's identity. The author thanks Soumik Pal for continued guidance during this project and providing helpful comments on clarifying several arguments in the paper. I would also like to thank the two anonymous referees for their careful reading and detailed comments on the presentation and proofs, including an idea for simplifying the proof of Lemma \ref{lem:8}. 

\section{Preliminaries}

\subsection{Forest Constructions}\label{sec:forest}

In this section we describe the forest ``picture'' underlying some of our weak convergence arguments later. This model is based on a model used by Aldous in \cite{Aldous_AF} and then by Duquesne in \cite{Duquesne_CRTI}. However, since we will not be proving convergence of the contour functions, we do use a labeling which is convenient for our analysis and is not useful for any type of genealogical analysis which can be found in \cite{Duquesne_CRTI}.

In this work, a forest will mean a locally finite planar graph on the vertex set 
$$
\{w_j, m_j: j = 0,1,\dotsm\}
$$ which has a finite number of connected components which themselves are rooted planar trees. The vertices of the form $w_j$ will be called the \textit{non-mutant vertices} (depicted in Figure \ref{fig:intro} as black vertices) and the vertices of the form $m_j$ will be called \textit{mutant vertices} (depicted as the white unlabeled vertices in Figure \ref{fig:intro}). When referring to a forest $\f$, we will let $\mathscr{V}(\f)$ denote the vertex set of $\f$ and $\mathscr{E}(\f)$ denote the edge set.

We will construct our random forests in a breadth-first approach from two given sequences of non-negative integers $(\chi_j; j = 0,1,\dotsm)$ and $(\eta_j; j=1,2,\dotsm)$ along with a number $k = 0,1,2,\dotsm$. At each height there will be two types of vertices, mutant and non-mutant vertices. There will only be one mutant vertex in each generation, and will be omitted from most of the objects we count on our forest. Children of mutant vertices will appear last in the breadth-first ordering when restricting to each height.

Each $\chi_j$ will represent the number of offspring the $j^\text{th}$ non-mutant vertex will have. Each $\eta_j$ will represent the number of immigrants that will appear in generation $j$. The immigrants will be non-mutant offspring of the mutant vertex from the previous generation.

We start with defining $\f_0$ to be the graph with vertex set $$\mathscr{V}(\f_0) = \{w_0,\dots,w_{k-1},m_0\}$$ and edge set $\mathscr{E}(\f_0) = \emptyset$. The vertex $m_0$ will be the mutant vertex. These vertices will be the roots of the random forest, and will consequently be the vertices of height $0$. 
For each $h\ge 1$ we construct $\f_h$ from $\f_{h-1}$ by adding vertices of height $h$. There are two cases which we must consider: (1) there are no non-mutant vertices at height $h-1$ in $\f_{h-1}$ or (2) the non-mutant vertices at height $h-1$ in $\f_{h-1}$ are $\{w_j:j = a,a+1,\dotsm, b-1\}$ for some natural numbers $a<b$.

In Case (1), we add $\eta_h$ non-mutant vertices and a single mutant vertex as offspring of the mutant vertex $m_{h-1}$. The collection of non-mutant vertices of $\f_{h-1}$ is of the form 
$$
\{w_j: j = 0,1,\dotsm, b-1\}
$$ for some $b \ge 0$. When $b = 0$ the above set is empty. We let $\f_h$ be the graph with vertex set $\mathscr{V}(\f_h):=\mathscr{V}(\f_{h-1})\cup \{w_j: j= b,b+1,\dotsm, b+ \eta_h-1\}\cup \{m_h\}$ and the edge set $$\mathscr{E}(\f_{h}):=\mathscr{E}(\f_{h-1}) \cup \{(m_{h-1}, w_j): j= b,b+1,\dotsm, b+\eta_h - 1\} \cup \{(m_{h-1},m_h)\}.$$
The height of the vertices of $\f_h$ will be the distance to the root, which inductively implies for all $v\in \f_h\setminus \f_{h-1}$ that $\h(v) = h$. We remark that if $\eta_h = 0$ then the only vertex in $\f_h\setminus\f_{h-1}$ is the vertex $m_h$. 

In Case (2), we add $\chi_j$ non-mutant vertices as offspring for each non-mutant vertex $w_j$ in generation $h-1$. We also add $\eta_h$ non-mutant vertices and a single mutant vertex as the offspring of $m_{h-1}$. We go into more detail in order to describe the breadth-first labeling. There are non-mutant vertices at height $h-1$ in $\f_{h-1}$. These vertices will be of the form $\{w_j:j = a,\dotsm, b-1\}$ for some $a<b$. Each vertex $w_j$ gives birth to $\chi_j$ children at height $h$ and the vertex $m_{h-1}$ will give birth to $\eta_h$ non-mutant children and 1 mutant child. Thus the vertex set for $\f_h$ will be
\begin{equation*}
\mathscr{V}(\f_h) = \mathscr{V}(\f_{h-1})\cup \left\{ w_j: j = b, b+1,\dotsm, b+ \sum_{i=a}^{b-1} \chi_i + \eta_h - 1 \right\}\cup\{m_h \}.
\end{equation*} Moreover, for each $\ell = 0,1,\dots, b-a-1$ we add edges between a vertex $w_{a+\ell}$ and $w_j$ for $j = b+\sum_{i=a}^{a+\ell-1} \chi_i , b+\sum_{i=a}^{a+\ell-1}\chi_i+1,\dotsm, b+ \sum_{i=a}^{a+\ell-1} \chi_i + \chi_{a+\ell} - 1$. We also add edges between $m_{h-1}$ and $m_h$ along with $w_j$ for $j = b+\sum_{i=a}^{b-1} \chi_i  , b+\sum_{i=a}^{b-1}\chi_i +1,\dotsm , \sum_{i=a}^{b-1} \chi_i -1 + \eta_h.$ Again, the vertices in $\f_{h}\setminus \f_{h-1}$ have distance $h$ from a root. 

We continue this process ad infinitum, and define $\f:=\bigcup_{h\ge 0} \f_h$. By construction, for $i<j$ we must have $\h(w_i)\le \h(w_j)$. 

\begin{definition}\label{def:GWI}
	A forest $\f$ constructed as above starting from $k$ non-mutant vertices is called \textbf{Galton-Watson immigration forest} (GWI forest for short) when the integer sequence $(\chi_j; j = 0,1,\dotsm)$ are i.i.d. and the sequence $(\eta_h; h = 1,2,\dotsm)$ are i.i.d. and are independent of $(\chi_j)$. We call the common distribution, say $\mu$, of $(\chi_j)$ the offspring distribution, and we call the common distribution, say $\nu$, of $(\eta_h)$ the immigration rate. We denote the law of $\f$ by $\GWI_k(\mu,\nu)$.  
\end{definition}

\subsection{Processes defined on the forest} \label{sec:processesOnForests}

In this section we define several processes on the forest, and describe some of the relationships between the various processes and the breadth-first labeling of the vertices. Let $\f$ be a forest constructed as above. We recall that the forest $\f$ is rooted, and that the height of a vertex is defined as the distance from the vertex to a root in the graph distance. 

Define the \textit{height profile} of the forest $\f$ as the process $Z^\f = \left(Z^\f_h: h \in \N_0 \right)$ by 
\begin{equation*}
Z_h^\f = \#\{\text{non-mutant vertices $v\in \f$ with $\h(v) = h$}\}.
\end{equation*} This counts the number of non-mutant vertices at height $h$ in the forest $\f$. We also define the \textit{cumulative height profile} of $\f$ as the process $C^\f = (C^\f_h; h\in \N_0)$ by
\begin{equation*}
C_h^\f = \sum_{j = 0}^h Z^\f_j,
\end{equation*} and its right-continuous inverse $V^\f = (V^\f_r; r = 0,1,\dotsm)$ by
\begin{equation*}
V_r^\f = \inf\{h\ge 0: C^\f_h >r\}.
\end{equation*}

We observe that $C_h^\f$ is the index of the first vertex at height $h+1$. This follows easily from the convention that we start indexing at 0. Hence, $V^\f_r = \h(w_r)$ and so $\displaystyle Z^\f_{V^\f_r}$ is the width of the layer of the forest containing the vertex $w_r$. 

Lastly, we refer the reader back to \eqref{eqn:KJdef} for a definition of the cumulative breadth-first cousin process $K^\f$ and the cumulative breadth-first height process $J^\f$.

\subsection{L\'evy processes}\label{sec:lp} We will provide a brief overview of spectrally positive L\'evy processes. More details and proofs of the statements below can be found in Chapter VII of Bertoin's monograph \cite{Bertoin_LP}.

A (possibly killed) spectrally positive L\'evy process $X = (X_t;t\ge 0)$ is a L\'evy process which contains no negative jumps. Its Laplace transform exists and uniquely characterizes the process $X$. The Laplace transform must be of the form
\begin{equation}\label{eqn:lpOfsplp}
\E\left[ \exp(-\lambda X_t) \right] = \exp(t\Psi(\lambda)) \qquad \forall \lambda>0.
\end{equation} Moreover, the function $\Psi$ must be of the form
\begin{equation}\label{eqn:psi}
\Psi(\lambda) = -\kappa + \alpha\lambda +\beta\lambda^2 + \int_{(0,\infty)} \left( e^{-\lambda r} -1 + \lambda r 1_{[r<1]} \right)\,\pi(dr)
\end{equation} where $\kappa, \beta\ge 0$, $\alpha\in \R$ and $\pi$ is a Radon measure on $(0,\infty)$ such that $\displaystyle \int_{(0,\infty)} (1\wedge r^2)\,\pi(dr)<\infty$. Conversely, for each such $\Psi$, there exists a spectrally positive L\'evy process with such a Laplace transform.

A particular class of spectrally positive L\'evy processes are subordinators. These are L\'evy processes with increasing sample paths. A subordinator $Y$ has a Laplace transform of the form
\begin{equation}\label{eqn:lpOfsub}
\E\left[ \exp(-\lambda Y_t)\right] = \exp(-t\Phi(\lambda)),\qquad \forall \lambda>0,
\end{equation} where $\Phi$ is of the form
\begin{equation}\label{eqn:phi}
\Phi(\lambda) = \kappa' + \alpha'\lambda - \int_{(0,\infty)}(e^{-\lambda r} -1)\,\nu(dr)
\end{equation} with $\kappa',\alpha'\ge 0$ and $\nu$ is a Radon measure with $\displaystyle \int_{(0,\infty)} (1\wedge r)\,\nu(dr)<\infty$. In our work, we concern ourselves with the case where $\Phi(\lambda) = \delta \lambda$ for some $\delta>0$, which makes the subordinator $Y_t = \delta t$. 

\subsection{Continuous state branching processes} \label{sec:CBI}

Continuous state branching processes with immigration (CBI processes) arise as scaling limits of discrete Galton-Watson processes, see \cite{KW_BPI}. A CBI process $Z = (Z_t; t\ge 0)$ is a Feller process on $[0,\infty]$ which is absorbed at $\infty$. We denote by $\E_x[-]$ the expectation conditionally given $Z_0 = x\ge 0$. As shown by \cite{KW_BPI}, the Laplace transform of $Z$ is of the form
\begin{equation*}
\E_x \left[\exp\{-\lambda Z_t\}  \right] = \exp \left[-xu(t,\lambda)-\int_0^t \Phi(u(s,\lambda))\,ds \right],\qquad  \forall \lambda>0
\end{equation*} where $u$ is the unique non-negative solution to the integral equation
\begin{equation*}
u(t,\lambda) + \int_0^t \Psi(u(s,\lambda))\,ds = \lambda,
\end{equation*} for functions $\Psi$ and $\Phi$. The function $\Psi$ is called the branching mechanism and must be of the form \eqref{eqn:psi} and the function $\Phi$ is called the immigration rate and must be of the form \eqref{eqn:phi}. Conversely, given any two such functions $\Psi$ and $\Phi$, there exists a CBI process with branching mechanism $\Psi$ and immigration rate $\Phi$. For simplicity, we will use $\CBI_x(\Psi,\Phi)$ to refer to the law of a CBI process starting from $x\ge 0$ and satisfies the above Laplace transform.

By \cite{KW_BPI}, we know that continuous state branching processes with immigration are in one-to-one correspondence with pairs of L\'evy processes $(X,Y)$ satisfying \eqref{eqn:lpOfsplp} and \eqref{eqn:lpOfsub}. The bijection described there is in terms of  the Laplace transforms of the respective processes. A path-wise identification does exist, thanks to the work of Caballero, P\'erez Garmendia and Uribe Bravo \cite{CGU_CSBPI}. As mentioned previously, the authors of \cite{CGU_CSBPI} show that if $X$ is a spectrally positive L\'evy process with Laplace exponent $(-\Psi)$ and $Y$ is an independent subordinator with Laplace exponent $\Phi$ then a c\`adl\`ag solution to \eqref{eqn:lamp} exists, is unique, and is a $\CBI_x(\Psi,\Phi)$ process. When $Y$ is identically 0, the L\'evy process $X$ in \eqref{eqn:lamp} is stopped upon hitting $-x$. The path-wise relationship when $Y = 0$ was observed by Lamperti \cite{Lamperti_CSBP1}, although proved later by Silverstein \cite{S_LT}. 

\subsection{The $\Psi$-height process} \label{sec:heightIntro} In this section we recall properties of a $\Psi$ height process $H = (H_t;t\ge 0)$. These processes were introduced by Le Gall and Le Jan in \cite{LL_BPLP}, and further examined in \cite{LL_BPLP2,DL_Levy}. The excursions of this process are the continuum random tree analog of the contour process for a discrete tree. We recall some of their properties, but do not endeavor to state things in full generality. To this end, we assume that 
\begin{equation}\label{eqn:psi2}
\Psi(\lambda) = \alpha \lambda+ \beta \lambda^2 + \int_{(0,\infty)} (e^{-\lambda r}-1 + \lambda r)\, \pi(dr)
\end{equation} with $\alpha, \beta\ge 0$ and $\pi$ a Radon measure with the stronger integrability condition $\int_0^\infty (r\wedge r^2)\,\pi(dr)<\infty$. We make the assumption that $\Psi$ is conservative, i.e. it satisfies equation \eqref{eqn:cons}, and that further $\Psi$ satisfies
\begin{equation}\label{eqn:ctsH}
\int_1^\infty \frac{1}{\Psi(u)}\,du<\infty.
\end{equation} Both of these assumptions on $\Psi$ are slight restrictions on a general theory, but they imply that a L\'evy process $X$ with Laplace exponent $(-\Psi)$ has paths of infinite variation, non-negative jumps, and does not drift towards $+\infty$.

Let $X$ be a L\'evy process with Laplace exponent $(-\Psi)$ satisfying the above restrictions. The value of the height process at time $t$ is a way to ``measure'' the size of the set
\begin{equation*}
\left\{s\le t: X_{s-} = \inf_{s\le r\le t} X_r \right\}.
\end{equation*} This is done through a time-reversal approach. For each $t>0$, define $\displaystyle \widehat{X}^{(t)} = \left(\widehat{X}^{(t)}_s; s\in[0,t] \right)$ by $\widehat{X}^{(t)}_s = X_t - X_{(t-s)-}$ and the corresponding supremum process by $\widehat{S}^{(t)}_s = \sup_{r\in[0,s]} \widehat{X}^{(t)}_r$. The value $H_t$ is a normalization of the local time at 0 of the process $\widehat{X}^{(t)}-\widehat{S}^{(t)}$. By \cite[Theorem 1.4.3]{DL_Levy} $H$ has a continuous modification if and only if the condition in \eqref{eqn:ctsH} is satisfied. Whenever \eqref{eqn:ctsH} is satisfied, we will assume that $H$ is this modification. For more information on height processes see \cite{DL_Levy}.

The process $H$ possesses a family of local times $L = (L^a_t;t\ge 0, a\ge0)$ which almost surely satisfies the occupation density formula:
\begin{equation}\label{eqn:od}
\int_0^t g(H_r)\,dr = \int_0^\infty g(a) L_t^a \,da, \qquad \forall g\in C_c(\R),\,t\ge 0.
\end{equation} See \cite[Proposition 1.3.3]{DL_Levy}. There is also a Ray-Knight theorem for these processes as shown in \cite[Theorem 1.4.1]{DL_Levy}. Namely, define $T_x = \inf\{t: L_t^0 = x\} = \inf\{t:X_t = -x\}$ then 
\begin{equation*}
\left(L_{T_x}^a; a \ge 0 \right) \sim \CBI_x(\Psi,0).
\end{equation*}
Similar Ray-Knight theorems were obtained by Warren \cite{Warren_RK} involving sticky Brownian motion.

The Ray-Knight theorem above was extended by Lambert \cite{Lambert_CBI} and Duquesne \cite{Duquesne_CRTI} to allow for some immigration. Lambert's work is slightly more general; however Duquesne's work contains some better approximation results for the local time. For $\delta>0$ and $x\ge 0$ define the left-height process $\cev{H} = (\cev{H}_t;t\ge0)$ by 
\begin{equation}\label{eqn:leftHeight}
\cev{H}_t = H_t + \frac{1}{\delta}(L_t^0 - x)_+.
\end{equation} By \cite{Duquesne_CRTI}, there exists a jointly measurable family of local times $\cev{L} =$\linebreak$ \left(\cev{L}^a_t ;t\ge 0, a\ge 0 \right)$ which is continuous and increasing in $t$ and satisfies equation \eqref{eqn:od} with $\cev{H}$ and $\cev{L}$ replacing $H$ and $L$. Since $\cev{H}_t\to \infty$ as $t\to\infty$, the limit $\cev{L}^a_\infty := \lim_{t\upto\infty} \cev{L}^a_t$ is finite almost surely due to \eqref{eqn:cons}. Moreover, the Ray-Knight theorem, see \cite[Theorem 1.2, Remark 3.2]{Duquesne_CRTI} and \cite[Theorem 5]{Lambert_CBI}, states
\begin{equation}\label{eqn:cevLCBI}
\left( \cev{L}^a_\infty ; a\ge 0\right) \sim \CBI_x(\Psi,\Phi),\qquad \Phi(\lambda) = \delta \lambda.
\end{equation} In the case where the branching process is a squared Bessel process and the height process is a reflected Brownian, a similar result was shown in \cite{LY_RKThms}.

\section{Integral Relationships for CBIs} \label{sec:BPFacts}

In this section we describe various properties of continuous state branching processes and their implications for random trees. We will mostly work under the assumption that $\Psi$ is a conservative branching mechanism; however, when discussing the implications for left-height processes we must work under slightly stricter assumptions. We recall that $\Psi$ being conservative is equivalent, see \cite{Grey_explosion}, to the almost sure finiteness of a $\CBI_x(\Psi,0)$ process. 

Let us start by gathering certain properties of processes related to a \linebreak$\CBI_x(\Psi,\Phi)$ process $Z$. If $\Phi(\lambda) = \delta\lambda$ for some $\delta>0$ then, by Theorem 2 in \cite{CGU_CSBPI}, we can and do assume $Z$ to be the unique c\`adl\`ag solution to
\begin{equation} \label{eqn:lamp2}
Z_t = x + X_{C_t} + \delta t,\qquad C_t = \int_0^t Z_s\,ds,
\end{equation} where $X$ is a L\'evy process with Laplace exponent $(-\Psi)$. Moreover, using this representation of $Z$ along with Lemma 3 and Corollary 5 in \cite{CGU_CSBPI} and the strong Markov property for $Z$ the following lemma is easy to see.
\begin{lemma}\label{lem:cLem}
	Let $Z$ be a $\CBI_x(\Psi,\Phi)$ process where $\Psi$ is a conservative branching mechanism (i.e. satisfies \eqref{eqn:cons} and \eqref{eqn:psi}) and $\Phi(\lambda) = \delta \lambda$ for some $\delta > 0$. Then $C_t:=\int_0^t Z_s\,ds$ is strictly increasing and almost surely $C_t\to \infty$ as $t\to\infty$. 
\end{lemma}

The above lemma tells us that the process $V$ defined in equation \eqref{eqn:vdef} is actually the two-sided inverse of $C$. Moreover, since $C$ diverges towards infinity almost surely, the value of $V_r$ is finite for each $r\ge0$. 

We now move to the proof of Lemma \ref{lem:integral1}. 
\begin{proof}[Proof of Lemma \ref{lem:integral1}]
	We assume that we are working on a probability space where $Z$ has the path-by-path representation in \eqref{eqn:lamp2}, for a L\'{e}vy process $X$. The change of variable formula implies
	\begin{equation*}
	\int_0^{V_r} F(u,Z_u)\,dC_u = \int_0^r F(V_u,Z_{V_u})\,du
	\end{equation*} for locally bounded functions $F$. Moreover, since $dC_u = Z_u\,du$, we can claim
	\begin{equation*}
	\int_0^{V_r} (Z_u-\delta u)Z_u \,du = \int_0^r (Z_{V_u} - \delta V_u)\,du.
	\end{equation*} However, by \eqref{eqn:lamp2} and the fact that $V$ is the two-sided inverse of $C$, 
	\begin{equation*}
	Z_{V_u} = x+ X_{C_{V_u}} + \delta V_u  = x+X_u + \delta V_u.
	\end{equation*}
	The result desired claim now easily follows. 
\qed  \end{proof}

In fact the above proof easily implies the following corollary as well.
\begin{corollary}
	Under the conditions for Lemma \ref{lem:integral1} and for any $n = 1,2,\dotsm$
	\begin{equation*}
	\int_0^{V_r} (Z_u-\delta u)^n Z_u\,du \overset{d}{=} \int_0^r \left(x+ X_u\right)^n\,du,
	\end{equation*} where $X$ is a L\'evy process with Laplace exponent $(-\Psi)$.
\end{corollary}

\subsection{The continuum random tree interpretation}\label{sec:crtSec}

Here we must strengthen the assumptions put onto $\Psi$ in order to guarantee that there is a continuous $\Psi$-height process. As discussed in Section \ref{sec:heightIntro}, we require that $\Psi$ satisfies \eqref{eqn:cons}, \eqref{eqn:psi2} and \eqref{eqn:ctsH}. 

Let $\cev{H} = (\cev{H}_t;t\ge 0)$ be defined as in \eqref{eqn:leftHeight} and let $\cev{L} = (\cev{L}^a_t; a\ge 0, t\ge 0)$ denote its local time. We know from \eqref{eqn:cevLCBI} that the process $\left( \cev{L}_\infty^a;a \ge 0\right)$ is a $\CBI_x(\Psi,\Phi)$ process where $\Phi(\lambda) = \delta \lambda$. Define $\displaystyle V_r = \inf\left\{ x \ge 0: \int_0^x \cev{L}_\infty^y \,dy >r\right\}$ and then Lemma \ref{lem:integral1} implies
\begin{equation*}
\left(\delta\int_0^{V_r} a \cev{L}^a_\infty \,da- \int_0^{V_r}\left( \cev{L}^a_\infty \right)^2\,da; r\ge 0\right) \overset{d}{=} \left(-xr - \int_0^r X_u\,du;r\ge0\right).
\end{equation*}
However, the occupation time formula implies that almost surely the left-most integral above can be recognized as
\begin{equation*}
\int_0^{V_r} a \cev{L}^a_\infty \,da = \int_0^\infty \cev{H}_t 1_{[\cev{H_t}\le V_r]}\,dt.
\end{equation*} Hence, we can conclude the following:
\begin{corollary}\label{cor:crtHeight}
	Let $\cev{H}$ be defined as in \eqref{eqn:leftHeight} with $\Psi$ satisfying \eqref{eqn:cons}, \eqref{eqn:ctsH} and \eqref{eqn:psi2} and let $\cev{L}$ denote its local time. Then
	\begin{equation*}
	\left(\delta\int_0^\infty \cev{H}_t 1_{[\cev{H}_t\le V_r]}- \int_0^{V_r} \left(\cev{L}_\infty^a\right)^2\,da ; r\ge 0 \right) \overset{d}{=} \left(-xr - \int_0^r X_u\,du;r \ge 0 \right)
	\end{equation*} where $X$ is a L\'evy process with Laplace exponent $(-\Psi)$.
\end{corollary}

\section{Weak Convergence} \label{sec:weakConvergence}

Throughout this section we will let $(\f^{(n)};n\ge 1)$ denote a sequence of forests where each $\f^{(n)}$ is a $\GWI_{[nx]}(\mu_n,\nu_n)$-forest. We also assume that $\mu_n$ and $\nu_n$ satisfy Assumption \ref{ass:1}, and $(\gamma_n;n\ge 1)$ is the sequence of integers specified therein.


With this we can now prove the following joint convergence lemma:
\begin{lemma}\label{lem:ZCV}
	If $\mu_n$ and $\nu_n$ satisfy Assumption \ref{ass:1}, then the following joint convergence holds in the product of the Skorokhod space $\D(\R_+,\R)^3$
	\begin{align}\label{eqn:ZCV}
	&\left( \left(\frac{1}{n} Z^{*,n}_{[\gamma_n s]} \right)_{s\ge 0}, \left( \frac{1}{n\gamma_n} C_{[\gamma_n t]}^{*,n}\right)_{t\ge 0}, \left(\frac{1}{\gamma_n}V^{*,n}_{[n\gamma_n r]} \right)_{r\ge 0}\right)\nonumber\\
	&\qquad \overset{(d)}{\Longrightarrow} \left( \left(Z_s \right)_{s\ge 0}, \left(\int_0^t Z_s\,ds \right)_{t\ge 0}, \left(\inf\left\{t: \int_0^t Z_s\,ds >r \right\} \right)_{r\ge 0} \right),
	\end{align} where $Z$ is a $\CBI_{x}(\Psi,\Phi)$ process.
\end{lemma}

\begin{proof}
	We remark that by \cite[Theorem 1.4]{Duquesne_CRTI}, the convergence of rescaling of $Z^{*,n}$ converges to a $\CBI_x(\Psi,\Phi)$ process, i.e.
	\begin{equation*}
	\left(\frac{1}{n}Z_{[\gamma_n s]}^{*,n};s\ge 0\right)\overset{(d)}{\Longrightarrow} \left( Z_s;s \ge 0\right).
	\end{equation*}
	
	We also observe 
	\begin{align*}
	\frac{1}{n\gamma_n} C_{[\gamma_nt]}^{*,n} &= \frac{1}{n\gamma_n} \sum_{h=0}^{[\gamma_nt]} Z^{*,n}_{h}\\
	&= \frac{1}{\gamma_n} \int_0^{[\gamma_nt]+1} \frac{1}{n} Z^{*,n}_{[u]}\,du\\
	&=\int_0^{([\gamma_n t]+1)/\gamma_n} \frac{1}{n} Z^{*,n}_{[\gamma_n s]}\,ds.
	\end{align*} Since $([\gamma_n t]+1)/\gamma_n \to t$ locally uniformly, the joint convergence
	\begin{align*}
	&\left( \left(\frac{1}{n} Z^{*,n}_{[\gamma_n s]} \right)_{s\ge 0}, \left( \frac{1}{n\gamma_n} C_{[\gamma_n t]}^{*,n}\right)_{t\ge 0}\right)\nonumber\\
	&\qquad \overset{(d)}{\Longrightarrow} \left( \left(Z_s \right)_{s\ge 0}, \left(\int_0^t Z_s\,ds \right)_{t\ge 0}\right),
	\end{align*} will follow once we argue the continuity of the map $f\mapsto (f,\int_0^\cdot f(s)\,ds)$. 
	
	To see that $f\mapsto \int_0^\cdot f(s)\,ds$ is continuous we prove that $\D$ continuously embeds into $L^1_{\loc}$. Indeed, suppose that $f_n\to f$ in $\D$ and let $t$ be a continuity point of $f$. For $g\in \D(\R_+,\R)$, set $\|g \|_{[0,t]} = \sup_{r\le t} |g(r)|$. Then there exists a sequence $\tau_n:[0,t]\to[0,t]$ such that $\|\tau_n-id\|_{[0,t]}\vee \|f\circ \tau_n - f_n\|_{[0,t]}\to 0$ as $n\to\infty$. Then
	\begin{align*}
	\int_0^t |f_n(s)-f(s)|\,ds&\le \int_0^t |f_n(s)-f(\tau_n(s))|\,ds + \int_0^t |f(\tau_n(s))-f(s)|\,ds\\
	&\le t \|f_n-f\circ\tau_n\|_{[0,t]} + \int_0^t |f(\tau_n(s))-f(s)|\,ds.
	\end{align*} The first term can easily be handled since $f_n\to f$ in $\D$ and the second term converges to 0 as $n\to\infty$ by dominated convergence. This proves the continuity of the embedding $\D\hookrightarrow L^1_{\loc}$. Since $f\mapsto \int_0^\cdot f(s)\,ds$ is a continuous map from $L^1_{\loc}(\R_+)\to C(\R_+)$ and $C(\R_+)$ embeds continuously into $\D$, we have shown the desired continuity.
	
	The convergence of the third coordinate in equation \eqref{eqn:ZCV} follows from \cite[Theorem 7.2]{Whitt_WC} and Lemma \ref{lem:cLem}. Indeed, define the set $E \subset \D(\R_+,\R)$ to be the collection of functions $f$ which are unbounded from above with $f(0)\ge 0$ and equip this set with the Skorokhod $J_1$ topology. Then the map $E\to E$ defined by $(f(t);t\ge0)\mapsto (\inf\{t\ge 0:f(t)>r\};r \ge 0)$ is measurable and it is continuous on the set of strictly increasing functions. 
\qed \end{proof}

\subsection{Proofs of Theorem \ref{thm:weakConvergence} and Theorem \ref{thm:excurForest}}\label{sec:someProofs}

\begin{proof}[Proof of Theorem \ref{thm:weakConvergence}] Throughout this proof we refer to $Z$ as  a $\CBI_x(\Psi,\Phi)$ process, the quantity $C_t = \int_0^t Z_s\,ds$ and $V$ defined in \eqref{eqn:vdef}. In what follows we sometimes write the index in the process as $Z^{*,n}(h)$ instead of as a subscript $Z^{*,n}_h$, and similar remarks hold for the other processes.

	We recall that the first index of a vertex in $\f^{(n)}$ at height $h+1$ is $C_h^{*,n}$.
	Therefore
	\begin{align*}
	K^{*,n}(C^{*,n}_h) &= \sum_{j=0}^{C^{*,n}_h - 1} \csn(w_j) = \sum_{\ell = 0}^{h} Z^{*,n}_\ell (Z^{*,n}_\ell - 1)\\
	&= \sum_{\ell=0}^h \left(Z^{*,n}_\ell\right)^2 - C^{*,n}_h.
	\end{align*} Similarly, we can see that $$J^{*,n}(C^{*,n}_h) = \sum_{\ell=0}^h \ell Z_{\ell}^{*,n}.$$
	
	Consequently, 
	\begin{align*}
	\frac{1}{n^2\gamma_n}K^{*,n}(C^{*,n}_{[\gamma_n t]}) &= \frac{1}{n^2 \gamma_n} \sum_{\ell=0}^{[\gamma_n t]} \left(Z_{\ell}^{*,n}\right)^2 - \frac{1}{n^2\gamma_n} C^{*,n}_{[\gamma_n t]}\\
	&= \int_0^{([\gamma_nt]+1)/\gamma_n} \left(\frac{1}{n} Z^{*,n}_{[\gamma_n s]}\right)^2\,ds - \frac{1}{n^2\gamma_n} C^{*,n}_{[\gamma_nt]}
	\end{align*} and
	\begin{equation*}
	\frac{1}{n\gamma_n^2} J^{*,n}(C^{*,n}_{[\gamma_nt]}) = \int_0^{([\gamma_n t]+1)/\gamma_n} \frac{[\gamma_n s]}{\gamma_n} \cdot \frac{1}{n} Z^{*,n}_{[\gamma_ns]}\,ds.
	\end{equation*}
	This easily implies the weak convergence in $\D(\R_+,\R)^2$
	\begin{align*}
	&\left(\left(\frac{1}{n^2\gamma_n}K^{*,n}(C^{*,n}_{[\gamma_nt]}) ;t\ge 0\right), \left(\frac{1}{n\gamma_n^2}J^{*,n}(C^{*,n}_{[\gamma_nt]});t\ge 0 \right)\right)
	\\&\qquad \overset{(d)}{\Longrightarrow}\left( \left( \int_0^t Z_s^2\,ds;t\ge0\right), \left( \int_0^t sZ_s\,ds;t\ge0\right) \right).
	\end{align*} Moreover this convergence is joint with the convergence in \eqref{eqn:ZCV}, and hence by a lemma on pg. 151 in \cite{B_CPM}
	\begin{equation*}\begin{split}
	&\left(\frac{1}{n^2\gamma_n} K^{*,n}(C^{*,n}(V^{*,n}_{[n\gamma_n r]})) - \frac{\delta}{n\gamma_n^2} J^{*,n}(C^{*,n}(V^{*,n}_{[n\gamma_nr]})) ;r\ge 0\right) \\
	&\qquad\qquad\qquad\overset{(d)}{\Longrightarrow} \left( \int_0^{V_r} (Z_s-\delta s)Z_s\,ds;r\ge 0\right).
	\end{split}
	\end{equation*} By Lemma \ref{lem:integral1} and Slutsky's theorem, the result follows if we can show
	\begin{equation} \label{eqn:KdiffToZero}
	\left(\left|\frac{1}{n^2\gamma_n}K^{*,n}(C^{*,n}(V^{*,n}_{[n\gamma_n r]})) - \frac{1}{n^2\gamma_n}K^{*,n}_{[n\gamma_n r]}  \right|;r\ge 0\right) \overset{(d)}{\Longrightarrow} \mathbf{0}:=(0; r\ge 0)
	\end{equation} and
	\begin{equation*}
	\left(\left|\frac{1}{n\gamma^2_n}J^{*,n}(C^{*,n}(V^{*,n}_{[n\gamma_n r]})) - \frac{1}{n\gamma_n^2}J^{*,n}_{[n\gamma_n r]}  \right|;r\ge 0\right) \overset{(d)}{\Longrightarrow} \mathbf{0}.
	\end{equation*}
	
	To argue the above convergences, we observe that $C^{*,n}(V^{*,n}_r-1)\le r < C^{*,n}(V^{*,n}_r)$. Since $K^{*,n}$ is increasing, we have
	\begin{align}
	\bigg|\frac{1}{n^2\gamma_n}& K^{*,n}(C^{*,n}(V^{*,n}_{[n\gamma_n r]})) - \frac{1}{n^2\gamma_n} K^{*,n}_{[n\gamma_nr]} \bigg| \nonumber\\
	&\le \left|\frac{1}{n^2\gamma_n} K^{*,n}(C^{*,n}(V^{*,n}_{[n\gamma_n r]})) - \frac{1}{n^2\gamma_n} K^{*,n}(C^{*,n}(V^{*,n}_{[n\gamma_n r]}-1))  \right| \nonumber \\
	&\le \frac{1}{n^2\gamma_n} \left(Z^{*,n}(V^{*,n}_{[n\gamma_nr]})\right)^2 \label{eqn:KboundbyZ2}.
	\end{align} The presence of the square in the second inequality follows from the fact that $$K^{*,n}(C^{*,n}(h)) - K^{*,n}(C^{*,n}(h-1)) = Z_h^{*,n}\cdot (Z_h^{*,n}-1) \le \left( Z_{h}^{*,n}\right)^2.$$ The desired convergence in \eqref{eqn:KdiffToZero} follows from 
	\begin{equation*}
	\left(\frac{1}{n^2\gamma_n} \left(Z^{*,n}(V^{*,n}_{[n\gamma_n r]}) \right)^2;r\ge 0\right)\overset{(d)}{\Longrightarrow} \mathbf{0}.
	\end{equation*} The above convergence then holds by Lemma \ref{lem:ZCV} and standard weak convergence arguments for time-changes (see e.g. a lemma on pg. 151 in \cite{B_CPM}). Indeed, we have
	\begin{equation} \label{eqn:ZcompVconv}
	\left(	n^{-1} Z^{*,n} \left( V^{*,n}_{[n\gamma_n r]} \right); r\ge 0 \right) \overset{(d)}{\Longrightarrow} \left(Z\left( \inf\{t: \int_0^t Z_s\,ds >r \} \right); r\ge0\right).
	\end{equation} Since $\gamma_n\to \infty$ by Assumption \ref{ass:1}, the stated convergence to zero holds. 
	
	Reasoning as in \eqref{eqn:KboundbyZ2} we get the upper bound 
	\begin{equation*}
	\left|\frac{1}{n\gamma_n^2} J^{*,n}(C^{*,n}(V^{*,n}_{[n\gamma_n r]})) - \frac{1}{n\gamma_n^2} J^{*,n}_{[n\gamma_nr]} \right| \le \frac{1}{n\gamma_n^2}V^{*,n}_{[n\gamma_nr]} \cdot Z^{*,n}(V^{*,n}_{[n\gamma_nr]}).
	\end{equation*} By Lemma \ref{lem:ZCV} and the convergence in \eqref{eqn:ZcompVconv}, the following weak convergence holds
	\begin{equation*}
	\left(\frac{1}{n\gamma_n^2}V^{*,n}_{[n\gamma_nr]} \cdot Z^{*,n}(V^{*,n}_{[n\gamma_nr]}) ;r\ge0\right) \overset{(d)}{\Longrightarrow} \mathbf{0}.
	\end{equation*}  The result now follows.
\qed \end{proof}

Similar arguments can yield Theorem \ref{thm:excurForest}. However, the authors of \cite{GS_paper} and \cite{GLS_RBB} are interested in distributional properties of the quantity 
\begin{equation*}
\int_0^1 B^{|\br|,x}_t - \frac{1}{\beta}\int_0^\infty \left(L_{1,x}^v \right)^2\,dv, \quad \forall \beta>0
\end{equation*} where $B^{|\br|,x}$ is a reflected Brownian bridge conditioned on its local time at level zero and time 1 being exactly $x$ and $L_{1,x}^v$ is the local time of $B^{|\br|,x}$ at time 1 and level $v$. Hence, we prove the following proposition, which by taking $\beta = 2$ and using Theorem 1.13 in \cite{GLS_RBB} yields the formulation in Theorem \ref{thm:excurForest}.

\begin{proposition}
	Suppose that $\f^{(n)}$ is a uniformly chosen rooted labeled forest on $n$ vertices with $k_n$ roots where $2k_n/\sqrt{n}\to x\ge 0$ as $n\to\infty$. Then the following convergence in distribution holds
	\begin{equation*}
	\frac{1}{2 n^{3/2}} \sum_{v\in \f^{(n)}} \h(v) - \frac{2}{\beta n^{3/2}} \sum_{v\in \f^{(n)}} \csn(v) \overset{(d)}{\Longrightarrow} \int_0^1 B^{|\br|,x}_t\,dt - \frac{1}{\beta} \int_0^\infty (L_{1,x}^v)^2\,dv,	
	\end{equation*} where $B^{|\br|,x}$ and $L_{1,x}^v$ are as above.
\end{proposition} 

\begin{proof}
	This proof follows from arguments similar to Theorem \ref{thm:weakConvergence} and the results of Pitman \cite{Pitman_SDE} describing a conditional version of a result by Drmota and Gittenberger \cite{DG_Bridge}. Alternatively, when $x = 0$ we can use Jeulin's identity along with a prior work by Drmota and Gittenberger \cite{DG_profile}.
	
	Let $Z^{*,n} = (Z^{*,n}_h; h=0,1,\dotsm)$ denote the height profile of $\f^{(n)}$. Then, by Theorems 4 and 7 in \cite{Pitman_SDE}, we have
	\begin{equation*}
	\left(\frac{2}{\sqrt{n}} Z^{*,n}_{[2\sqrt{n}v]};v\ge0 \right)\overset{(d)}{\Longrightarrow} \left(L_{1,x}^v;v\ge 0\right),
	\end{equation*} where the convergence above is weak convergence in the Skorokhod space.
	
	Hence,
	\begin{align*}
	\frac{1}{n^{3/2}} \sum_{v\in \f^{(n)}} \h(v) &= \frac{1}{n^{3/2}}\sum_{h=0}^\infty hZ^{*,n}_{h}=\frac{1}{n^{3/2}}\int_0^\infty [u] Z^{*,n}_{[u]}\,du\\
	&= \int_0^\infty \frac{[2\sqrt{n}v]}{\sqrt{n}} \cdot \frac{2}{\sqrt{n}} Z^{*,n}_{[2\sqrt{n}v]} \,dv \overset{(d)}{\Longrightarrow} 2\int_0^\infty v L^{v}_{1,x}\,dv
	\end{align*} and, similarly
	\begin{align*}
	\frac{1}{n^{3/2}} \sum_{v\in \f^{(n)}} \csn(v) &= \frac{1}{n^{3/2}}\sum_{h=0}^\infty (Z^{*,n}_{h})(Z^{*,n}_h-1)\\
	&= \frac{1}{n^{3/2}} \sum_{h=0}^\infty\left\{(Z_h^{\ast,n})^2 - Z_h^{\ast,n}\right\}\\
	&= \frac{1}{2}\int_0^\infty \left(\frac{2}{\sqrt{n}}Z^{*,n}_{[2\sqrt{n}v]}\right)^2\,dv - \frac{1}{\sqrt{n}}\\
	&\overset{(d)}{\Longrightarrow}\frac{1}{2} \int_0^\infty (L_{1,x}^v)^2\,dv. 
	\end{align*} The result now easily follows.
\qed \end{proof}

\section{SDE Results} \label{sec:SDEstuff}

In this section we discuss the existence and uniqueness of solutions to SDEs of the form \eqref{eqn:Zsde} and \eqref{eqn:ZsdeUC}. We first study the situation with equation \eqref{eqn:Zsde} and then move onto studying equation \eqref{eqn:ZsdeUC}.

\subsection{Analysis of \eqref{eqn:Zsde}} \label{sec:ZC}

We begin by studying a fairly different stochastic differential equation. Let $\tilde{B}$ be a Brownian motion on the filtered probability space $(\Om,\F_t, \tilde{P})$ and let $Y_t$ be the unique strong solution to the stochastic differential equation
\begin{equation*}
dY_t = a d\tilde{B}_t + \left(\frac{c}{Y_t} - \frac{Y_t}{1-t}\right)\,dt,\qquad Y_0 = x\ge 0,\quad t\in[0,1]
\end{equation*} where $a>0$ and $c\ge \frac{a^2}{2}$ are constants. The process $Y$ is $a$ times a Bessel bridge of dimension $\delta = \frac{2c}{a^2}+1\ge 2$ and so a unique strong solution exists by properties of Bessel bridges. See Section XI.3 in \cite{RY} and \cite{GY_Bessel} for more properties about Bessel bridges.

Now let $f:\R_+\to \R$ be a continuous function and define $M$ to be the continuous martingale $M_t :=\frac{1}{a} \int_0^t f(s)\,d\tilde{B}_s$. Observe that $\mathcal{E}(M_t)$, where $\mathcal{E}$ is the Dol\'eans-Dade exponential, is a positive continuous local martingale. Hence, by Girsanov's theorem \cite[Theorem VIII.1.7]{RY}, the measures $P:= \mathcal{E}(M_t)\cdot \tilde{P}$ and $\tilde{P}$ are equivalent on $\F_t$ for each $t\ge 0$. Moreover, $B_t:= \tilde{B}_t - \frac{1}{a}\int_0^t f(s)\,ds$ is a $P$-Brownian motion and so
\begin{equation}\label{eqn:YsdeC}
dY_t = a dB_t + \left( \frac{c}{Y_t} + f(t) - \frac{Y_t}{1-t}\right)\,dt,\qquad Y_0 = x\ge 0,\quad t\in[0,1].
\end{equation}

Using \cite[Theorem IX.1.11]{RY}, we have the following lemma:
\begin{lemma}\label{lem:YCexist}
	There is weak existence and uniqueness in law to equation \eqref{eqn:YsdeC} for constants $a>0$, $c \ge \frac{a^2}{2}$ and continuous function $f$. 
	
	Moreover, the law of a solution $Y$ to \eqref{eqn:YsdeC} is equivalent to the law of $aR = (aR_t;t\in[0,1])$ where $R$ is a Bessel bridge of dimension $\frac{2c}{a^2}+1$. 
\end{lemma}

Using facts about Bessel bridges and the equivalence of measures described in the above lemma, we obtain the following corollary. See, for example, \cite[Chapter XI]{RY}, \cite{GY_Bessel}. In particular, the derivation of Equation (2.6) in \cite{Hariya} generalizes for $\delta$-dimensional Bessel processes when $\delta>2$. 
\begin{corollary}\label{cor:YCpp}
	Suppose $Y$ is a solution to \eqref{eqn:YsdeC} with $c\ge a^2/2$ and $f$ a continuous function. Then the following hold almost surely.
	\begin{enumerate}
		\item $\displaystyle \int_0^1 \frac{1}{Y_t}\,dt <\infty$ if $c> \frac{a^2}{2}$ and $\displaystyle \int_0^t \frac{1}{Y_s}\,ds<\infty$ for all $t\in[0,1)$ when $x>0$ and $c = \frac{a^2}{2}$.
		\item $\displaystyle\int_0^1 \frac{Y_t}{1-t}\,dt <\infty$ if $c>\frac{a^2}{2}$.
		\item $Y_t > 0$ for all $t\in (0,1)$.
		\item The process $V_t:= \int_0^t \frac{1}{Y_s}\,ds$ for $t\in[0,1]$ is strictly increasing and bounded for $c>\frac{a^2}{2}$. It is strictly increasing and locally bounded on $[0,1)$ when $c = \frac{a^2}{2}$ and $x>0$.
	\end{enumerate}
\end{corollary}

When $c>\frac{a^2}{2}$, we can define the right-continuous inverse $\tau_v := \inf\{t: V_t>v\}$ for $v\in[0,V_1)$. Since $V$ is strictly increasing $\tau$ is actually a two-sided inverse of $V$. Define the process $Z_v = Y_{\tau_v}$ and observe on the event $\{V_1>t\}$ and for $v\in[0,t)$
\begin{align*}
Z_v &= x+ aB_{\tau_v}+ \int_0^{\tau_v} \left(\frac{c}{Y_s} + f(s) - \frac{Y_s}{1-s}\right)\,ds\\
&= x+ aB_{\tau_v} + \int_0^{v}\left(\frac{c}{Y_{\tau_u}} + f(\tau_u) - \frac{Y_{\tau_u}}{1-\tau_u}\right)\,d\tau_u.
\end{align*} We observe that $$
d\tau_u = Y_{\tau_u}\,du = Z_u\,du 
$$ and, by Proposition V.1.5 and Theorem V.1.6 in \cite{RY}, we can write $B_{\tau_v} = \int_0^v\sqrt{Z_v}\,dW_v$ for a Brownian motion $W$. Hence, for $v\in[0,V_1)$ 
\begin{equation}\label{eqn:Zint}
Z_v = x + \int_0^v a\sqrt{Z_v}\,dW_v +\int_0^v \left(c+ f\left(\int_0^u Z_s\,ds \right)Z_u  - \frac{Z_u^2}{ 1- \int_0^u Z_s\,ds}\right)\,du.
\end{equation} 
Finally, we observe that $\lim_{v\upto V_1} Z_v = \lim_{t\upto 1} Y_t = 0$ almost surely.
Similarly, given a process $Z$ satisfying \eqref{eqn:Zint} the process $Y_t = Z_{V_t}$ with $V_t = \inf\{u: \int_0^u Z_s\,ds >t\}$ for $t\in[0,1)$ solves the stochastic differential equation \eqref{eqn:YsdeC}.

We have thus argued the following proposition:
\begin{proposition}\label{prop:existZC}
	For $f$ a continuous function there is weak existence for the stochastic differential equation \eqref{eqn:Zsde} when $c>\frac{a^2}{2}$ and $x\ge 0$. 
	
	Moreover, if $Z$ is any such solution then $Y_t = Z_{V_t}$ where $V_t = \inf\{u: \int_0^u Z_s\,ds >t\}$ for $t\in[0,1)$ solves \eqref{eqn:YsdeC}.
\end{proposition}

\subsection{Analysis of \eqref{eqn:ZsdeUC}} \label{sec:ZUC}

Throughout this section we assume that $f:\R_+\to \R$ is a continuous function and $g:\R_+\to \R$ is a function such that $g^2$ is Lipschitz and there is some $\eps>0$ and $M<\infty$ such that $\eps\le g \le M$. 

We begin by observing that since $g^2$ is Lipschitz and bounded below, the function $\frac{1}{g^2}$ is Lipschitz. Now define the function $h(t) = \inf\{v: \int_0^v g^2(s)\,ds>t\}$, and observe that this function is continuous since $g^2$ is strictly positive and moreover, it is the unique solution to $$ h'(t) = \frac{1}{g^2(h(t))},\qquad h(0) = 0.$$ Observe that for each $c\ge 0$ the function $\tilde{b}(t) = \frac{2c}{g^2(t)}+1$ is bounded and continuous. If we define
$$
b(t) = \tilde{b}(h(t)),
$$ then $b(t)$ is also continuous. 

Since $\tilde{b}$ is continuous, by \cite[Proposition 1]{FG_bbvd} and \cite[Theorem IX.3.5]{RY}, for each $c\ge 0$ there exists a unique strong solution to the stochastic differential equation
\begin{equation*}
dX_t = 2\sqrt{X_t}\,d\tilde{B}_t + {b}(t)\,dt ,\qquad X_0 = x\ge 0, 
\end{equation*} where $\tilde{B}$ is a standard Brownian motion on some filtered probability space $(\Om,\F_t,\tilde{P})$. But this means that there is a weak solution to the stochastic differential equation
\begin{equation*}\begin{split}
dX^{(1)}_t &= 2\sqrt{X_t^{(1)}} \,d\tilde{B}_t +\tilde{b}(X^{(2)}_t)\,dt, \qquad X_0^{(1)} =x \ge 0\\
dX^{(2)}_t &= \frac{1}{g^2(X^{(2)}_t)}dt,\hspace{1.05in} X_0^{(2)} = 0,
\end{split}
\end{equation*} since $b(t) = \tilde{b}(h(t))$ and $X^{(2)}_t = h(t)$. Since $g$ satisfies $\eps\le g\le M$, by \cite[Proposition IX.1.13]{RY}, there is a weak solution to the stochastic differential equation:
\begin{equation*}
\begin{split}
dX^{(1)}_t &= 2g(X^{(2)}_t)\sqrt{X_t^{(1)}} \,d\tilde{B}_t +g^2(X^{(2)}_t)\tilde{b}(X^{(2)}_t)\,dt, \qquad X_0^{(1)} =x \ge 0\\
dX^{(2)}_t &= dt,\hspace{2.2in} X_0^{(2)} = 0.
\end{split}
\end{equation*} But this means that $X^{(1)}$ solves
\begin{equation}\label{eqn:Xref1}
dX_t = 2g(t)\sqrt{X_t}\,d\tilde{B}_t + \left(2c+ g^2(t) \right)\,dt,\qquad X_0=x\ge 0.
\end{equation}

The next lemma states that the process $X$ can be bounded below by a deterministic time-change of a squared Bessel process $S = (S(t);t\ge 0)$. We recall \cite{RY} that a squared Bessel process of dimension $\delta\ge 0$ and starting from $x \ge 0$ is the unique strong solution of the stochastic differential equation
\begin{equation*}
dS(t) = 2\sqrt{S(t)}\,dB_t + \delta \,d t ,\qquad S(0) = x,
\end{equation*} for a standard Brownian motion $B$.
\begin{lemma}
	Suppose $X$ is a solution to \eqref{eqn:Xref1} with respect to a Brownian motion $\tilde{B}$ and started from $x\ge0$. Fix any $0\le \delta\le\inf_{t} \left(\frac{2c}{g^2(t)}+1\right)$. Then on the same probability space $(\Om,\F,P)$ there exists a $\delta$-dimensional squared Bessel process $S = (S(t);t\ge0)$ starting from $x$ such that
	$$
	P\left(	S\left(\int_0^t g^2(r)\,dr\right) \le X_t, \forall t \right) = 1.
	$$
\end{lemma}
\begin{proof}
	We prove this lemma by a time-change. Let $\tau_t = \inf\{ s: \int_0^s g^2(r)\,dr>t\}$ and define the process $R$ by $
	R_{t} = X_{\tau_t}.
	$ Observe
	\begin{align*}
	R_t &= x+ \int_0^{\tau_t}2 g(r) \sqrt{X_r}\,d\tilde{B}_r + \int_0^{\tau_t} \left(2c+g^2(r) \right)\,dr\\
	&= x+ M_{\tau_t} + \int_0^t \left(\frac{2c}{g^2(\tau_s)}+1\right)\,ds
	\end{align*} where $M_t = \int_0^t2g(r)\sqrt{X_r}\,d\tilde{B}_r$ and we used the change of variable $r = \tau_s$ in the drift integral. We now observe
	$$
	\langle M_{\tau_\cdot} \rangle_{t} = \langle M \rangle_{\tau_t} = \int_0^{\tau_t} 4g^2(r) {X_r}\,dr =\int_0^t 4 X_{\tau_s}\,ds =  \int_0^t 4 R_s\,ds,
	$$ and so by the Dambis-Dubins-Schwarz theorem \cite[Theorem V.1.6]{RY} there is a Brownian motion $B$ such that $$
	R_t = x+ \int_0^t 2\sqrt{R_s}\,dB_s + \int_0^t \left(\frac{2c}{g^2(\tau_s)}+1\right)\,ds.
	$$
	For each $\delta\ge0$ and with respect to this Brownian motion $B$, there exists a unique strong solution to the stochastic differential equation
	\begin{equation*}
	dS_t = 2\sqrt{S_t}\,dB_t + \delta\,dt,\qquad S_t = x.
	\end{equation*} The comparison theorems \cite[Theorem IX.3.7]{RY} imply that $$
	P(S_t \le R_t,\forall t) = 1
	$$ for any $0\le \delta\le \inf_{t} \left(\frac{2c}{g^2(t)} +1\right)$. The desired claim now follows.
\qed \end{proof}

Arguments similar to Lemma \ref{lem:YCexist} and Corollary \ref{cor:YCpp} give the following lemma, the details of which are omitted:
\begin{lemma}\label{lem:YsquareExistence} Suppose $f$ is a continuous function and $g^2$ is a Lipschitz function with $0<\eps\le g\le M<\infty$. There exists a weak solution to the stochastic differential equation:
	\begin{equation*}
	dX_t = 2g(t)\sqrt{X_t} \,dB_t + \left( 2c+g^2(t) + 2f(t)\sqrt{X_t} \right)\,dt,\qquad X_0 = x\ge0.
	\end{equation*}
	
	Moreover, for any such solution the following hold almost surely:
	\begin{enumerate}
		
		\item If $c>\frac{1}{2}\sup_{t} g^2(t)$ and $x \ge 0$ or $c = \frac{1}{2}\sup_{t} g^2(t)$ and $x>0$, then $\displaystyle \int_0^t X_s^{-1/2}\,ds <\infty$ for each $t<\infty$.
		
		\item If $c\ge \frac{1}{2}\sup_{t} g^2(t)$ and $x > 0$ then $\inf\{t: X_t = 0\} = \infty$.

	\end{enumerate}	
	
\end{lemma} 

Observe that if $X_0 = x>0$ then $X_t$ almost surely never reaches $0$. Therefore, we can apply It\^o's rule to $Y_t = \sqrt{X_t}$ and see
\begin{align*}
dY_t &= \frac{1}{2\sqrt{X_t}} dX_t - \frac{1}{8X^{3/2}_t} d\langle X\rangle_t\\
&= g(t) dB_t + \left( \frac{c+ \frac{1}{2}g^2(t)}{\sqrt{X_t}} + f(t) \right)\,dt - \frac{1}{8 X_t^{3/2}} (4 g^2(t) X_t)\,dt\\
&= g(t)dB_t + \left(\frac{c}{Y_t} +f(t) \right)\,dt.
\end{align*} Hence we have argued the following lemma:
\begin{lemma} \label{lem:YUCpp} Suppose $f$ is a continuous function and $g^2$ is a Lipschitz function with $0<\eps\le g\le M<\infty$. Assume that $c\ge \sup_t \frac{1}{2} g^2(t)$ and $x>0$. There exists a weak solution to the stochastic differential equation
	\begin{equation}\label{eqn:YsdeUC}
	dY_t = g(t)\,dB_t + \left(\frac{c}{Y_t} + f(t) \right)\,dt,\qquad Y_0=x>0.
	\end{equation}
	
	Moreover, any such solution is strictly positive and so the process $V_t = \int_0^t \frac{1}{Y_s}\,ds$ is continuous and strictly increasing. 
\end{lemma}

Finally, we use a time change and obtain the following proposition.
\begin{proposition}\label{prop:existZUC}
	Suppose $f$ is a continuous function and $g^2$ is a Lipschitz function with $0<\eps\le g\le M<\infty$. Assume that $c\ge \sup_t \frac{1}{2} g^2(t)$ and $x>0$. Then there exists a weak solution to \eqref{eqn:ZsdeUC}. Moreover, for any such solution $Z$, the process $C_t = \int_0^t Z_s\,ds$ is strictly increasing.
\end{proposition}
\begin{proof}
	The proof of existence is omitted, since it follows from the arguments similar to those in Proposition \ref{prop:existZC}. 
	
	We now argue that $C_t$ is strictly increasing.  To this end let $(Z,C)$ be a solution to \eqref{eqn:ZsdeUC} with respect to some Brownian motion $W$. To argue that $C$ is strictly increasing, we argue its derivative $Z$ is strictly positive. To this end, we argue by contradiction and suppose that $\tau = \inf\{t: Z_t = 0\}<\infty$. First $\tau>0$ by continuity and the fact $Z_0 = x>0$. Hence, $Z_t>0$ for $t\in[0,\tau)$ and hence $C_t$ is strictly increasing for $t\in[0,\tau)$ and we set $h = C_\tau>0$.
	
	We now define the right-continuous inverse of $C$ as $V_t = \inf\{r: C_r>t\}$ and define $Y_t = Z_{V_t}$. Observe that for $t\in [0,h)$ we have
	\begin{align*}
	Y_t &= x+ \int_0^{V_t} g(C_u) \sqrt{Z_u}\,dW_u+ \int_0^{V_t} (c+ f(C_u)Z_u) \,du\\
	&= x + \tilde{M}_{V_t} + \int_0^t \left(\frac{c}{Z_{V_u}} + f(u)\right) \,dt\\
	&= x+ \tilde{M}_{V_t} + \int_0^t \left(\frac{c}{Y_u}+f(u) \right)\,dt
	\end{align*} where $\tilde{M}_t = \int_0^t g(C_u)\sqrt{Z_u}\,dW_u$. 
	
	Moreover, we have that $$
	\langle\tilde{M}_{V_\cdot}\rangle_t = \langle\tilde{M}\rangle_{V_t}  = \int_0^{V_t} g^2(C_u) Z_u\,du = \int_0^t g^2(t)\,dt.
	$$ Hence, by Dambis-Dubins-Schwarz \cite[Theorem V.1.6]{RY} there is a Brownian motion $B$ on the interval $[0,h)$ such that
	$$
	Y_t = x+ \int_0^{t}g(u)\,dB_u + \int_0^t \left(\frac{c}{Y_u} + f(u)\right)\,ds,\qquad t\in[0,h).
	$$
	
	Moreover, $V_t = \int_0^t \frac{1}{Y_s}\,ds$ is continuous and strictly increasing by Lemma \ref{lem:YUCpp} and the stochastic process $Y$ is strictly positive and uniformly bounded away from zero on bounded time intervals. Hence $$
	Z_{\tau} = \lim_{r\upto h} Z_{V_r} = \lim_{r\upto h} Y_r > 0.
	$$ This gives the desired contradiction and we conclude that $Z_t$ is always strictly positive and hence $C_t = \int_0^t Z_u\,du $ is strictly increasing.
\qed \end{proof}

\section{Proofs of Normality Results}\label{sec:intRel}

\subsection{Proof of Theorem \ref{thm:bridgeLT}}

Throughout this subsection we fix a $c>\frac{a^2}{2}$ and $x\ge 0$ and let $Y$ and $Z$ be related as in Section \ref{sec:ZC}. In particular, we know that $Y$ and $Z$ solve \eqref{eqn:YsdeC} and \eqref{eqn:Zsde}, respectively, and are related by $Y_t = Z_{V_t}$ where $V_t = \int_0^t \frac{1}{Y_s}\,ds$. As we have seen the right-continuous inverse of $V_t$ is equal to $\int_0^t Z_u\,du$. 

\begin{proof}[Proof of Theorem \ref{thm:bridgeLT}] The proof is quite similar to the proofs in \cite[Theorem 1.1]{Hariya} and \cite[Theorem 1.11]{GS_paper}.
	
	Since $Y$ solves \eqref{eqn:YsdeC}, we observe
	\begin{align*}
	\int_0^{1} Y_u\,du &= x+ \int_0^{1} a B_u\,du + \int_0^{1} \left(\int_0^u \left\{\frac{c}{Y_s} +f(s) - \frac{ Y_s}{1-s} \right\}\,ds \right)\,du\\
	&= x+\int_0^{1} aB_t \,dt + \int_0^{1}cV_u\,du+ \int_0^{1} \left(\int_0^u \left\{f(s) - \frac{ Y_s}{1-s} \right\}\,ds \right)\,du.
	\end{align*}
	Corollary \ref{cor:YCpp} allows us to apply Fubini's theorem to the $\frac{Y_s}{1-s}$ integrand to obtain:
	\begin{equation*}
	\int_0^1 \int_0^u \frac{Y_s}{1-s}\,ds\,du = \int_0^1 ds\left\{\int_s^1 \frac{Y_s}{1-s}\,du \right\}  = \int_0^1 Y_s\,ds.
	\end{equation*}
	
	Hence we have
	\begin{equation*}
	\int_0^1 (2Y_u - cV_u)\,du = x+a\int_0^1 B_s\,ds + \int_0^1 \int_0^u f(s)\,ds\,du
	\end{equation*}
	
	It follows that 
	$$
	\int_0^1 \left(2Z_{V_t} - cV_t\right)\,dt \overset{d}{=} \mathscr{N}\left(\mu,\sigma^2 \right)
	$$ where $\mu = x+ \int_0^1 (1-s)f(s)\,ds$ and $\sigma^2 = \frac{a^2}{3}$.
	By the observation that $Z_v = 0$ for all $v\ge V_1$, and a change of variable we get
	$$
	\int_0^\infty (2Z_v-c v)Z_v \,dv \overset{d}{=}\mathscr{N}\left(\mu,\sigma^2 \right), 
	$$ which gives the desired claim.
\qed \end{proof}

\subsection{Proof of Theorem \ref{thm:gaussianStructure} }

\begin{proof}[Proof of Theorem \ref{thm:gaussianStructure}]
	
	We observe that given a solution $(Z,C)$ to \eqref{eqn:ZsdeUC}, we can define $Y_t = Z_{V_t}$ where $V_t$ satisfies \eqref{eqn:vdef} and then $Y$ is a weak solution to \eqref{eqn:YsdeUC}. This follows from the proof of Proposition \ref{prop:existZUC}.
	
	Thus
	\begin{equation*}
	dY_t = g(t)\,dB_t+ \left(\frac{c}{Y_t}+f(t) \right)\,dt, \qquad Y_0 = x.
	\end{equation*} We recall that $\displaystyle V_t = \int_0^t \frac{1}{Y_s}\,ds$ since $$
	dV_t = \frac{1}{Z_{V_t}}\,dt = \frac{1}{Y_t}\,dt.
	$$ Also, $V_t = \int_0^t \frac{1}{Y_s}\,ds$ is finite for each $t\in[0,\infty)$ by Lemma \ref{lem:YUCpp}. Then for each $t\ge 0$ we have 
	\begin{align*}
	\int_0^t Y_s\,ds &= xt+ \int_0^t\int_0^s g(r)\,dB_r ds + \int_0^t \int_0^s \left(\frac{c}{Y_r}+f(r)\right)\, dr\,ds\\
	&= xt + \int_0^t \int_0^s g(r)\,dB_r\,ds + \int_0^t cV_r \,dr + \int_0^t \int_0^s f(r)\,dr\,ds.
	\end{align*} Hence
	$$
	\int_0^t (Y_s - c V_s) \,ds = xt + \int_0^t \int_0^s g(r)\,dB_r\,ds + \int_0^t \int_0^s f(r)\,dr\,ds.
	$$
	
	We observe that the left-hand side is equal to 
	$$\int_0^t \left(Z_{V_s} -cV_s\right)\,ds = \int_0^{V_t} (Z_v - cv)Z_v\,dv=: X_t.
	$$
	
	We just need to show that $$
	X_t = xt + \int_0^t \int_0^s f(r)\,dr\,ds + \int_0^t \int_0^s g(r)\,dB_r\,ds
	$$ has the desired mean and covariance structure. It is easy to see
	$$
	\E[X_t] = xt+ \int_0^t \int_0^s f(r)\,dr\,ds = xt + \int_0^t (t-s)f(s)\,ds,
	$$ while the covariance structure follows from the following lemma.
	\qed \end{proof}

\begin{lemma}\label{lem:8}
	Let $g\in L^2_{\loc}(\R_+)$, let $B$ be a standard Brownian motion and define $G(h) = \int_0^h \int_0^ u g(s)\,dB_s\,du.$ Then $G(h) = \int_0^h (h-u)\,g(u)\,dB_u$, and consequently it is a centered Gaussian process such that for $h_1\le h_2$,
	$$
	\E[G(h_1)G(h_2)] = \int_0^{h_1} (h_2-s)(h_1-s)g^2(s)\,ds.
	$$
\end{lemma}
\begin{proof}
	The claim that $G(h) = \int_0^h (h-u) g(u)\,dB_u$ is simply the stochastic Fubini theorem \cite[Theorem 65]{Protter.05}. 
	
	The covariance structure now follows by It\^{o}'s isometry: for any $f_1,f_2\in L^2_{\loc}(\R_+)$
	\begin{align*}
	\E\left[\int_0^{h_1} f_1(s)\,dB_s \int_0^{h_2} f_2(t)\,dB_t\right] &= \int_0^{h_1\wedge h_2} f_1(s)\,f_2(s)\,ds.
	\end{align*} The result follows letting $f_1(u) = (h_1-u)g(u)$ and $f_2(u) = (h_2 - u)g(u)$.
\qed \end{proof}

\end{document}